\documentclass[11pt]{amsart} 

\usepackage{amssymb,tensor,enumitem,verbatim,theoremref,mathtools}
\usepackage[bookmarksdepth=3]{hyperref}
\usepackage[margin=1in]{geometry}
\usepackage[all,cmtip]{xy}
\usepackage[sort,noadjust]{cite}
\usepackage{nicematrix}
\usepackage{ytableau}
\usepackage[color=yellow]{todonotes}

\newcommand{\sm}[1]{[ \begin{smallmatrix} #1 \end{smallmatrix} ]}
\newcommand{\pmat}[1]{\begin{pmatrix} #1 \end{pmatrix}}
\newcommand{\subgrp}[1]{\langle #1 \rangle}
\newcommand{\Set}[1]{\left\{ #1 \right\}}
\newcommand{\set}[1]{\left\{ #1 \right\}}
\newcommand{\smallset}[1]{\{ #1 \}}

\newcommand{\abs}[1]{| #1 |}

\newcommand{\ol}[1]{\overline{#1}}

\newcommand{\floor}[1]{\lfloor #1 \rfloor}

\newcommand{\da}[1]{\!\!\downarrow_{#1}}

\DeclareMathOperator{\ad}{ad}
\DeclareMathOperator{\Ann}{Ann}

\DeclareMathOperator{\End}{End}
\DeclareMathOperator{\Grp}{Grp}
\DeclareMathOperator{\Hom}{Hom}
\DeclareMathOperator{\id}{id}

\DeclareMathOperator{\Ind}{Ind}
\DeclareMathOperator{\Lie}{Lie}
\DeclareMathOperator{\otr}{otr}
\DeclareMathOperator{\res}{res}
\DeclareMathOperator{\rad}{rad}
\DeclareMathOperator{\Res}{Res}
\DeclareMathOperator{\sgn}{sgn}
\DeclareMathOperator{\Span}{span}
\DeclareMathOperator{\tr}{tr}

\newcommand{\zero}{\ol{0}}
\newcommand{\one}{\ol{1}}

\newcommand{\C}{\mathbb{C}}
\newcommand{\N}{\mathbb{N}}
\newcommand{\T}{\mathbb{T}}
\newcommand{\Z}{\mathbb{Z}}

\newcommand{\CA}{\C A}
\newcommand{\CD}{\C D}
\newcommand{\CG}{\C G}
\newcommand{\CH}{\C H}
\newcommand{\CR}{\C R}
\newcommand{\CS}{\C S}

\newcommand{\calA}{\mathcal{A}}

\newcommand{\calP}{\mathcal{P}}
\newcommand{\calW}{\mathcal{W}}

\newcommand{\Pbar}{\ol{\calP}}

\newcommand{\Wbar}{\ol{\calW}}

\newcommand{\fa}{\mathfrak{a}}
\newcommand{\fD}{\mathfrak{D}}
\newcommand{\ff}{\mathfrak{f}}
\newcommand{\g}{\mathfrak{g}}
\newcommand{\gzero}{\g_{\zero}}
\newcommand{\gone}{\g_{\one}}
\newcommand{\gl}{\mathfrak{gl}}
\newcommand{\fh}{\mathfrak{h}}
\newcommand{\fs}{\mathfrak{s}}
\newcommand{\fsl}{\mathfrak{sl}}
\newcommand{\fq}{\mathfrak{q}}
\newcommand{\sq}{\mathfrak{sq}}

\newcommand{\glmn}{\gl(m|n)}

\newcommand{\Wmus}{W_{\mu_s}}
\newcommand{\Wmuszero}{W_{\mu_s,\zero}}
\newcommand{\Wmusone}{W_{\mu_s,\one}}

\newcommand{\Vone}{V_{\ol{1}}}
\newcommand{\Vzero}{V_{\ol{0}}}

\numberwithin{equation}{subsection}

\newtheorem{theorem}{Theorem}[subsection]

\newtheorem{lemma}{Lemma}[subsection]
\makeatletter\let\c@lemma\c@theorem\makeatother

\makeatletter\let\c@alg\c@theorem\makeatother

\newtheorem{proposition}{Proposition}[subsection]
\makeatletter\let\c@proposition\c@theorem\makeatother

\makeatletter\let\c@conj\c@theorem\makeatother

\newtheorem{corollary}{Corollary}[subsection]
\makeatletter\let\c@corollary\c@theorem\makeatother

\theoremstyle{definition}
\newtheorem{definition}{Definition}[subsection]
\makeatletter\let\c@definition\c@theorem\makeatother

\newtheorem{convention}{Convention}[subsection]
\makeatletter\let\c@convention\c@theorem\makeatother

\newtheorem{example}{Example}[subsection]
\makeatletter\let\c@example\c@theorem\makeatother

\newtheorem{remark}{Remark}[subsection]
\makeatletter\let\c@remark\c@theorem\makeatother

\makeatletter\let\c@fact\c@theorem\makeatother

\makeatletter\let\c@note\c@theorem\makeatother

\numberwithin{equation}{subsection}

\usepackage[capitalise]{cleveref}
\crefname{theorem}{Theorem}{Theorems}
\crefname{fact}{Fact}{Facts}
\crefname{note}{Note}{Notes}
\crefname{lemma}{Lemma}{Lemmas}
\crefname{alg}{Algorithm}{Algorithms}
\crefname{remark}{Remark}{Remarks}
\crefname{example}{Example}{Examples}
\crefname{prop}{Proposition}{Propositions}
\crefname{conj}{Conjecture}{Conjectures}
\crefname{convention}{Convention}{Convention}
\crefname{cor}{Corollary}{Corollaries}
\crefname{definition}{Definition}{Definitions}
\crefname{equation}{\!\!}{\!\!} 

\title{The Lie superalgebra of transpositions}

\author{Christopher M.\ Drupieski}
\address{Department of Mathematical Sciences,
		DePaul University,
		Chicago, IL 60614, USA}
\email{c.drupieski@depaul.edu}

\author{Jonathan R.\  Kujawa}
\address{Department of Mathematics \\
		Oregon State University \\
		Corvallis, OR 97331, USA}
\email{jonathan.kujawa@oregonstate.edu}

\thanks{Corresponding Author:  Jonathan R.\ Kujawa (kujawaj@oregonstate.edu)}

\date{\today}


\allowdisplaybreaks

\begin{document}

\begin{abstract}
We consider the group algebra of the symmetric group as a superalgebra, and describe its Lie subsuperalgebra generated by the transpositions. 
\end{abstract}


\maketitle

\section{Introduction}

\subsection{Questions of WunderNatur}

This paper answers a series of questions originally posed by the MathOverflow user WunderNatur in August 2022 \cite{WunderNatur:2022}: Considering the group algebra of the symmetric group $\CS_n$ as a superalgebra (by considering the even permutations in $S_n$ to be of even superdegree and the odd permutations in $S_n$ to be of odd superdegree), and considering $\CS_n$ as a Lie superalgebra via the super commutator,
	\[
	[x,y] = xy - (-1)^{\ol{x}\cdot \ol{y}} yx,
	\]
what is the structure of $\CS_n$ as a Lie superalgebra, and what is the structure of the Lie subsuperalgebra of $\CS_n$ generated by the transpositions? The non-super analogues of these questions were previously answered by Marin \cite{Marin:2007}; we describe Marin's motivation in Section~\ref{subsec:Marin}.

\subsection{Main results and methods}

Fix an integer $n \geq 2$, and let $\calP(n)$ denote the set of all integer partitions of $n$. Given $\lambda \in \calP(n)$, write $\lambda'$ for the partition that is conjugate (or transpose) to $\lambda$, let $S^\lambda$ be the simple Specht module for $\CS_n$ labeled by $\lambda$, and set $f^\lambda = \dim_\C(S^\lambda)$.

Let $\Pbar(n)$ be any fixed set of representatives in $\calP(n)$ for the equivalence relation generated by $\lambda \sim \lambda'$. Up to parity change, the simple $\CS_{n}$-super\-modules are labeled by the elements of $\Pbar(n)$, and are described as ungraded $\CS_n$-modules by
	\[
	W^{\lambda} = \begin{cases}
	S^{\lambda} & \text{if $\lambda = \lambda'$ (Type M)}, \\
	S^{\lambda} \oplus S^{\lambda'} & \text{if $\lambda \neq \lambda'$ (Type Q)}.
	\end{cases}
	\]
In the terminology of \cite{Brundan:2002}, the supermodule $W^\lambda$ is absolutely irreducible (i.e., irreducible as an ordinary $\CS_n$-module) if $\lambda = \lambda'$, and is self-associate (hence is naturally equipped with an odd involution $J^\lambda: W^\lambda \to W^\lambda$) if $\lambda \neq \lambda'$. Since $\C S_{n}$ is semisimple as a superalgebra, the graded version of the Artin--Wedderburn theorem provides a corresponding direct sum decomposition of $\CS_n$ into matrix superalgebras of types M and Q:
	\[
	\C S_{n} \cong \big[ \bigoplus_{\substack{\lambda \in \Pbar(n) \\ \lambda  = \lambda'}} \End_\C\left(W^{\lambda} \right) \big] \oplus \big[ \bigoplus_{\substack{\lambda \in \Pbar(n) \\ \lambda \neq  \lambda'}} Q\left( W^{\lambda} \right) \big];
	\]
see \cref{cor:CSn-as-superalgebra}. This of course also describes $\C S_{n}$ as a Lie superalgebra:
	\[
	\C S_{n} \cong \big[ \bigoplus_{\substack{\lambda \in \Pbar(n) \\ \lambda  = \lambda'}} \gl\left( W^{\lambda} \right) \big] \oplus \big[ \bigoplus_{\substack{\lambda \in \Pbar(n) \\ \lambda \neq  \lambda'}} \fq\left( W^{\lambda} \right) \big].
	\]

Write $\g_{n}$ for the Lie subsuperalgebra of $\CS_{n}$ generated by the transpositions, let $T_{n}$ be the sum in $\CS_n$ of all transpositions, and let $\fD (\C S_{n} )$ be the derived subsuperalgebra of the Lie super\-algebra $\CS_{n}$.  Then \cref{theorem:finaltheorem} states that
	\[
	\g_n = \fD(\CS_n) + \C T_n,
	\]
where
	\[
	\fD(\CS_n) \cong \big[ \bigoplus_{\substack{\lambda \in \Pbar(n) \\ \lambda  = \lambda'}} \fsl \left( W^{\lambda} \right) \big] \oplus \big[ \bigoplus_{\substack{\lambda \in \Pbar(n) \\ \lambda \neq  \lambda'}} \sq \left( W^{\lambda} \right) \big].
	\]
Here $\fsl(W^\lambda)$ denotes the special linear Lie superalgebra on the superspace $W^\lambda$, and $\sq(W^\lambda)$ denotes the subspace of $\fq(W^\lambda)$ of elements whose `odd trace' is zero; see \eqref{eq:sq(n)-supermatrices}.

The description of $\CS_{n}$ as a Lie superalgebra is straightforward from the classification of the simple $\CS_{n}$-supermodules. Sections \ref{S:Prelims} and \ref{S:SnSupergroup} gather together various results on super representation theory and then apply them to $\C S_{n}$ to produce this classification, as well as to show that each $W^\lambda$ admits a restriction to $\CS_{n-1}$ that is multiplicity free if one accounts for parity shifts; see \cref{cor:Wlambda-restriction} and \cref{rem:multiplicity-free} for details. Many of the results in Sections \ref{S:Prelims} and  \ref{S:SnSupergroup} can be found in the literature and are certainly not surprising to experts.

It takes considerably more effort to confirm the claimed description of $\g_{n}$.  The argument is by a ``grand loop'' induction on $n$, wherein the results of Sections \ref{S:Imageofgn}--\ref{subsec:gn-structure} are proved sequentially for the value $n$ under the assumption that the results in these sections have already been proved for the value $n-1$. The arguments require intricate calculations and considerable case-by-case analysis. An important role is played by the Gelfand--Zeitlin bases for the $S^{\lambda}$ given by the simultaneous eigenvectors for the action of the Jucys--Murphy elements.

\subsection{The results of Marin} \label{subsec:Marin}

The questions answered here were first considered by Marin \cite{Marin:2003,Marin:2007} in the classical (non-super) setting. As an ungraded algebra, $\CS_{n}$ is again a direct sum of matrix algebras thanks to the classical Artin--Wedderburn theorem. Marin showed that the Lie subalgebra of $\CS_n$ generated by the transpositions is reductive with semisimple part isomorphic to a direct sum of special linear, orthogonal, and symplectic Lie algebras. In particular, he showed that the transpositions generate a Lie algebra that is roughly half the dimension of the Lie super\-algebra $\g_{n}$. Thus, the graded and classical settings are quite different.

One of Marin's motivations was the representation theory of the braid group, $B_{d}$. For example, representations of the Type A Iwahori--Hecke algebra, $H_d(q)$, can be inflated to $B_d$ via a canonical surjective algebra homomorphism $\C B_d \to H_d(q)$. The algebra $H_d(q)$ does not have a natural coproduct and the tensor product of two $H_d(q)$-modules,  $V \otimes W$, is not in general again a module over $H_d(q)$.  However, it is a module for the braid group via the coproduct on $\C B_d$. Marin showed that the decomposition of $V \otimes W$ into simple $\C B_d$-modules could be determined from the Lie algebra of transpositions. Marin also calculated the algebraic envelope of the braid group in the simple representations that arise via inflation through the map $\C B_d \to H_d(q)$.

The representation theory of the braid group is a rich area of study with connections to topology, combinatorics, algebraic geometry, and categorification. The braid group admits evident $\Z$- and $\Z_{2}$-gradings (defined by declaring the generators to be of degree $1$ or $\one$, respectively), but as far as we are aware the graded representation theory of the braid group is rather neglected. While this paper focuses on the questions raised by WunderNatur, it does suggest that the graded representation theory of the braid group should be notably different from the classical setting and worth further study. For example, if one considers the algebra $\calA = \C [q,q^{-1}]$ as a superalgebra where $q$ is declared to be of odd superdegree (i.e., if we consider $\calA$ as a superalgebra via reduction modulo two of the $\Z$-grading which makes $\calA$ a graded field), then the Iwahori--Hecke algebra defined over $\mathcal{A}$, $H_{d}(q)_{\mathcal{A}}$, is a superalgebra when the generators are taken to be of odd superdegree. There is a surjective superalgebra homomorphism from $\mathcal{A}B_{d}$ to $H_{d}(q)_{\mathcal{A}}$ and it would be interesting to study the supermodules for the braid group afforded by this map.  

\subsection{Acknowledgements} \label{subsec:Acks}  The authors would like to thank the referee for their close reading and helpful comments.

\section{Preliminaries}\label{S:Prelims}

\subsection{Conventions}

Set $\Z_2 = \Z/2\Z = \{ \zero,\one \}$. Following the literature, we use the prefix `super' to indicate that an object is $\Z_2$-graded. We denote the decomposition of a vector superspace into its $\Z_2$-homogeneous components by $V = \Vzero \oplus \Vone$, calling $\Vzero$ and $\Vone$ the even and odd subspaces of $V$, respectively, and writing $\ol{v} \in \Z_2$ to denote the superdegree of a homogeneous element $v \in \Vzero \cup \Vone$. If we state a formula in which homogeneous degrees of elements are specified, we mean that the formula is true as written for homogeneous elements, and that it extends by linearity to non-homogeneous elements. When written without additional adornment, we consider the field $\C$ to be a superspace concentrated in even super\-degree. All superspaces are assumed to be vector spaces over the field $\C$, all linear maps are $\C$-linear, and except when indicated by a modifier (e.g., `Lie'), all superalgebras are assumed to be associative and unital. Given a superspace $V$, let $\dim(V) = \dim_{\C}(V)$ be the ordinary dimension of $V$ as a $\C$-vector space.

A linear map between superspaces is \emph{even} if it preserves homogeneous degrees, and is \emph{odd} if it reverses homogeneous degrees. Given superspaces $V$ and $W$, let $\Hom(V,W) = \Hom_{\C}(V,W)$ be the superspace of all $\C$-linear maps $\phi: V \to W$, and let $\End(V) = \Hom_{\C}(V,V)$. Let $V^* = \Hom(V,\C)$ be the usual linear dual of $V$. In general, isomorphisms between superspaces will be denoted by `$\cong$' and, except when stated otherwise, should be understood as arising via even linear maps.  We write `$\simeq$' rather than `$\cong$' to emphasize when an isomorphism arises via an odd linear map.

\begin{remark}
Our convention for the use of the symbols `$\simeq$' and `$\cong$' is different then in \cite{Kleshchev:2005}. In the spirit of Robert Recorde, our choice of notation is motivated by our point of view that objects that are even-isomorphic are ``more equal'' than objects that are isomorphic by an odd or inhomogenous isomorphism.
\end{remark}

Given a superspace $V$, let $\Pi(V) = \set{ v^\pi : v \in V}$ be its parity shift. As a superspace, $\Pi(V)_{\zero} = \Vone$ and $\Pi(V)_{\one} = \Vzero$, with $\ol{v^\pi} = \ol{v} + \one$. Then $(-)^\pi : v \mapsto (-1)^{\ol{v}} v^\pi$ defines an odd isomorphism $V \simeq \Pi(V)$.

Given a superalgebra $A$ and (left) $A$-super\-modules $M$ and $N$, we say that a linear map $f: M \to N$ is an $A$-supermodule homomorphism if $f(a.m) = (-1)^{\ol{a} \cdot \ol{f}} a.f(m)$ for all $a \in A$ and $m \in M$, and we write $\Hom_A(M,N)$ for the set of all $A$-supermodule homomorphisms from $M$ to $N$. The parity shift $\Pi(M)$ of an $A$-supermodule is again an $A$-supermodule, with action defined by $a.m^\pi = (a.m)^\pi$. Then the function $(-)^\pi : m \mapsto (-1)^{\ol{m}} m^\pi$ is an odd $A$-supermodule isomorphism $M \simeq \Pi(M)$.

Let $\N = \set{0,1,2,3,\ldots}$ be the set of non-negative integers.

\subsection{Semisimple superalgebras} \label{subsec:simple-superalgebras}

Most of the material in this section comes from \cite[\S2]{Brundan:2002} and \cite[\S3.1]{Cheng:2012}. For the authors' benefit, we write out some of the details that were left to the reader in \cite{Brundan:2002,Cheng:2012}. As in \cite{Brundan:2002,Cheng:2012}, we make the standing assumption that each superalgebra is finite-dimensional.

A superalgebra $A$ is \emph{simple} if it has no nontrivial superideals.

\begin{example}[Type M simple superalgebras] \label{ex:M(m|n)}
Given a finite-dimensional superspace $V$, the endomorphism algebra $\End(V)$ is a simple superalgebra. Fixing a homogeneous basis for $V$, and making the identification $V \cong \C^{m|n} := \C^m \oplus \Pi(\C^n)$ for some $m,n \in \N$ via this choice of basis, $\End(V)$ identifies with the matrix superalgebra
	\[
	M(m|n) := \Set{ \left[	\begin{array}{c|c}
								A & B \\
								\hline
								C & D
								\end{array} \right] : A \in M_m(\C), B \in M_{m \times n}(\C), C \in M_{n \times m}(\C), D \in M_n(\C)}.
	\]
As an ungraded associative algebra, $M(m|n) = M_{m+n}(\C)$.
\end{example}

\begin{example}[Type Q simple superalgebras] \label{ex:Q(n)}
Let $V$ be a finite-dimensional vector superspace equipped with an odd involution $J : V \to V$; i.e., an odd linear map such that $J \circ J = \id_V$. Then
	\begin{equation} \label{eq:Q(V)}
	Q(V) = Q(V,J) = \set{ \theta \in \End(V) : J \circ \theta = \theta \circ J}
	\end{equation}
is a simple subsuperalgebra of $\End(V)$. Fix a basis $\set{v_1,\ldots,v_n}$ for $\Vzero$, and set $v_i' = J(v_i)$ for $1 \leq i \leq n$, so that $\set{v_1',\ldots,v_n'}$ is a basis for $\Vone$. Via this choice of homogeneous basis, one has $V \cong \C^{n|n}$ and $Q(V)$ identifies with the set of supermatrices
	\begin{equation} \label{eq:Q(n)}
	Q(n) := \Set{ \left[	\begin{array}{c|c}
							A & B \\
							\hline
							B & A
							\end{array} \right] : A \in M_n(\C), B \in M_n(\C) }.
	\end{equation}
As an ungraded associative algebra, $Q(n) \cong M_n(\C) \oplus M_n(\C)$ via the map $\sm{A & B \\ B & A} \mapsto (A+B,A-B)$.
\end{example}

\begin{remark}
In the literature, the definition \eqref{eq:Q(V)} is frequently stated with the requirement that the graded commutator $J \circ \theta - (-1)^{\ol{\theta}} \cdot \theta \circ J$ be equal to $0$, rather than the requirement that the ordinary commutator $J \circ \theta - \theta \circ J$ be equal to $0$. We find it more convenient to use the version stated here. Through appropriate choices of homogeneous bases, both versions admit the matrix realization \eqref{eq:Q(n)}. For related discussion, see \cite[\S 1.1.4]{Cheng:2012}.
\end{remark}

Given an associative superalgebra $A$, let $\abs{A}$ denote the underlying associative algebra obtained by forgetting the superspace structure on $A$. Let
	\[
	Z(A) = \{ a \in A : ab = (-1)^{\ol{a} \cdot \ol{b}} ba \text{ for all } b \in A \}
	\]
be the graded center of $A$ (i.e., the center in the sense of superalgebras), and let
	\[
	Z(\abs{A}) = \set{ a \in A : ab = ba \text{ for all } b \in A}
	\]
be the ungraded center of $A$ (i.e., the center in the ordinary, non-super sense). Then $Z(A)$ and $Z(\abs{A})$ are each subsuperspaces of $A$, i.e., $Z(B) = Z(B)_{\zero} \oplus Z(B)_{\one}$ for $B \in \left\{A, \abs{A}\right\}$. Also note that $Z(A)_{\zero} = Z(\abs{A})_{\zero}$.

\begin{example}
Let $m,n \in \N$.
	\begin{enumerate}
	\item $Z(M(m|n)) = Z(M(m|n))_{\zero} = Z(\abs{M(m|n)})$, spanned by the identity matrix $I_{m|n}$.
	\item $Z(Q(n))_{\zero}$ is spanned by the identity matrix $I_{n|n}$.
	\item $Z(Q(n))_{\one} = 0$, but $Z(\abs{Q(n)})_{\one}$ is nonzero, spanned by the `odd identity matrix' $\sm{0 & I_n \\ I_n & 0}$.
	\end{enumerate}
\end{example}

\begin{theorem}[{\cite[Theorem 3.1]{Cheng:2012}}]
Let $A$ be a finite-dimensional simple associative superalgebra.
	\begin{enumerate}
	\item If $Z(\abs{A})_{\one} = 0$, then $A \cong M(m|n)$ for some $m,n \in \N$.
	\item If $Z(\abs{A})_{\one} \neq 0$, then $A \cong Q(n)$ for some $n \in \N$.
	\end{enumerate}
\end{theorem}

\begin{definition}[Type M and Q simple supermodules] \label{def:Type-MQ-simples}
Let $A$ be an associative super\-algebra and let $V$ be a simple $A$-supermodule, i.e., an $A$-supermodule having no proper $A$-sub\-super\-modules. Then either $V$ is simple as an $\abs{A}$-module, in which case $V$ is said to be of \emph{Type~M} (or \emph{absolutely irreducible}, in the terminology of \cite{Brundan:2002}), or else $V$ is reducible as an $\abs{A}$-module, in which case $V$ is said to be of \emph{Type Q} (or \emph{self-associate}, in the terminology of \cite{Brundan:2002}).
\end{definition}

Given a superspace $V$, let $\pi_V: V \to V$ be the parity automorphism, defined by
	\[
	\pi_V(v) = (-1)^{\ol{v}} v.
	\]
In particular, $\pi_A: A \to A$ is a superalgebra automorphism. A subspace $U$ of a vector super\-space $V$ is a sub\emph{super}space of $V$ if and only if $\pi_V(U) = U$. Given an $\abs{A}$-module $U$, let $\pi_A^*(U)$ be the $\abs{A}$-module obtained by pulling back the module structure along $\pi_A$. Thus for $a \in A$ and $u \in U$, one has
	\[
	a. \pi_A^*(u) = (-1)^{\ol{a}} \cdot \pi_A^*(a.u).
	\]
If $U$ is an $\abs{A}$-submodule of an $A$-supermodule $V$, then $\pi_V(U)$ is also an $\abs{A}$-submodule of $V$, and the map $\pi_V(U) \to \pi_A^*(U)$, $\pi_V(u) \mapsto \pi_A^*(u)$, is an $\abs{A}$-module isomorphism. In particular, for each $A$-super\-module $V$, one has $V = \pi_V(V) \cong \pi_A^*(V)$ as $A$-supermodules.

\begin{lemma}[{\cite[Lemma 2.3]{Brundan:2002}}] \label{lemma:self-associate}
Let $V$ be a finite-dimensional simple $A$-super\-module of Type Q, and let $U$ be a proper simple $\abs{A}$-submodule of $V$. Then as an $\abs{A}$-module,
	\[
	V = U \oplus \pi_V(U) \cong U \oplus \pi_A^*(U),
	\]
with $U \not\cong \pi_V(U)$ as $\abs{A}$-modules, and the homogeneous subspaces of $V$ are 
	\[
	\Vzero = \set{ u + \pi_V(u) : u \in U} \quad \text{and} \quad \Vone = \set{ u - \pi_V(u) : u \in U}.
	\]
In particular, if $u_1,\ldots,u_n$ is a basis for $U$, then
	\[
	\set{u_1+\pi_V(u_1),\ldots,u_n+\pi_V(u_n)} \quad \text{and} \quad \set{u_1-\pi_V(u_1),\ldots,u_n-\pi_V(u_n)}
	\]
are bases for $\Vzero$ and $\Vone$, respectively.

The linear map $J = J_V: V \to V$, defined for $u \in U$ by $J(u \pm \pi_V(u)) = u \mp \pi_V(u)$, is an $\abs{A}$-module homo\-morphism. Considered as a function $J: V \to \Pi(V)$, $u \pm \pi_V(u) \mapsto [u \mp \pi_V(u)]^\pi$, the map $J$ is an even $A$-supermodule isomorphism $V \cong \Pi(V)$.
\end{lemma}

\begin{proof}
Most of the details of the proof are given in \cite{Brundan:2002}, though one point that is not explicitly explained is the fact that $U \not\cong \pi_V(U)$. Here is a justification for this statement. Let $\pi = \pi_V$.

Suppose for the sake of argument that there exists an $\abs{A}$-module isomorphism $\psi: U \to \pi(U)$. Let $\phi = \pi \circ \psi : U \to U$. Then also $\pi \circ \phi = \psi$, because $\pi \circ \pi = \id_V$, and $\phi$ is a linear bijection such that for all $a \in A$ and $u \in U$ one has $\phi(a \cdot u) = (-1)^{\ol{a}} a \cdot \phi(u)$. Consequently, $\phi^2 : U \to U$ is an $\abs{A}$-module isomorphism, so by Schur's Lemma it is a nonzero scalar multiple of the identity. Rescaling $\phi$ if necessary, we may assume that $\phi^2 = \id_U$.

Now since $V = U \oplus \pi(U)$ and $V = \Vzero \oplus \Vone$, it follows that also  $V = U^+ \oplus U^-$, where
	\begin{align*}
	U^+ &= \set{ \phi(u) + \pi(u) : u \in U} = \set{ u + \pi(\phi(u)) : u \in U}, \quad \text{and} \\
	U^- &= \set{ \phi(u) - \pi(u) : u \in U} = \set{ u - \pi(\phi(u)) : u \in U}.
	\end{align*}
For $u \in U$, the decomposition of $\phi(u) + \pi(u)$ into its even and odd components is
	\begin{align*}
	\phi(u) + \pi(u) &= \Big( \tfrac{1}{2} [ \phi(u) + \pi(\phi(u)) ] + \tfrac{1}{2} [ \phi(u) - \pi(\phi(u)) ] \Big) + \Big( \tfrac{1}{2} [ u + \pi(u) ] - \tfrac{1}{2} [ u - \pi(u) ] \Big) \\
	&= \tfrac{1}{2} \Big( [ \phi(u) + \pi(\phi(u)) ] + [ u + \pi(u) ] \Big) + \tfrac{1}{2} \Big( [ \phi(u) - \pi(\phi(u)) ] - [ u - \pi(u) ] \Big) \\
	&= \tfrac{1}{2} \Big( [ \phi(u) + \pi(u) ] + [ u + \pi(\phi(u)) ] \Big) + \tfrac{1}{2} \Big( [ \phi(u) + \pi(u) ] - [ u + \pi(\phi(u)) ] \Big).
	\end{align*}
After the second and third equals signs, the expressions within the big parentheses are homogeneous of even and odd superdegree, respectively. This shows that $U^+$ is a subsuperspace of $V$. Finally, for $a \in A$ one has
	\[
	a \cdot [ \phi(u) + \pi(u) ] = (-1)^{\ol{a}} \cdot [ \phi(a \cdot u) + \pi(a \cdot u) ],
	\]
so $U^+$ is a proper $A$-subsupermodule of $V$. In a similar fashion, $U^-$ is a proper $A$-subsupermodule of $V$. Then $V$ is a direct sum of two proper sub\-super\-modules, a contradiction.
\end{proof}

\begin{remark}
The decomposition of a Type Q simple $A$-supermodule into a direct sum of non-iso\-morphic simple $\abs{A}$-modules is canonical, by the uniqueness of isotypical components.
\end{remark}

\begin{lemma}[Super Schur Lemma]
Let $V$ be a finite-dimensional simple $A$-supermodule. Then
	\[
	\End_{\abs{A}}(V) = \begin{cases}
	\Span\set{ \id_V} & \text{if $V$ is of Type M,} \\
	\Span\set{ \id_V, J_V} & \text{if $V$ is of Type Q,}
	\end{cases}
	\]
where $J_V$ is defined as in \cref{lemma:self-associate}. In particular, if $V$ is of Type Q, then $J_V$ is the unique $\abs{A}$-module homomorphism (up to scalar multiples) that is homogeneous of odd superdegree.
\end{lemma}

\begin{proof}
If $V$ is of Type M, the lemma is true by the classical Schur's Lemma. If $V$ is of Type Q, the classical Schur's Lemma gives $\End_{\abs{A}}(V) = \Span \{ \id_U, \id_{\pi(U)} \}$, with notation as in \cref{lemma:self-associate}. Since $\id_V = \id_U + \id_{\pi(U)}$ and $J_V = \id_U - \id_{\pi(U)}$, the result follows.
\end{proof}

\begin{remark} \label{rem:J_V-understood}
Henceforward, if $V$ is a finite-dimensional simple $A$-supermodule of Type Q, we will write $Q(V)$ to denote $Q(V,J_V)$.
\end{remark}

An $A$-supermodule $V$ is \emph{semisimple} if every subsupermodule of $V$ is a direct summand, or equivalently, if $V$ is a (direct) sum of simple $A$-supermodules.

\begin{theorem}[Super Artin--Wedderburn Theorem {\cite[Theorem 3.3]{Cheng:2012}}] \label{theorem:super-wedderburn}
The following statements are equivalent for a finite-dimensional associative superalgebra $A$:
	\begin{enumerate}
	\item Every $A$-supermodule is semisimple.
	\item The left regular $A$-module is a direct sum of minimal left superideals.
	\item The superalgebra $A$ is a direct sum of simple superalgebras. Specifically, if $\set{V_1,\ldots,V_n}$ is a complete, ir\-re\-dun\-dant set of simple $A$-super\-modules (up to homo\-geneous isomorphism), such that $V_1,\ldots,V_n$ are of Type M and $V_{m+1},\ldots,V_n$ are of Type Q, then the natural maps $A \to \End(V_i)$, arising from the $A$-supermodule structures on the $V_i$, induce a superalgebra isomorphism
		\[
		A \cong \left(\bigoplus_{i=1}^m \End(V_i) \right) \oplus \left( \bigoplus_{i=m+1}^n Q(V_i) \right).
		\]
	\end{enumerate}
A superalgebra that satisfies these conditions is called \emph{semisimple}.
\end{theorem}

\begin{lemma} \label{lemma:ss-super-ordinary}
Let $A$ be a finite-dimensional superalgebra. Then $A$ is semisimple (as a super\-algebra) if and only if $\abs{A}$ is semisimple (as an ordinary algebra).
\end{lemma}

\begin{proof}
If $A$ is a direct sum of simple superalgebras, then $\abs{A}$ is a direct sum of simple algebras, and hence is semisimple, by Examples \ref{ex:M(m|n)} and \ref{ex:Q(n)}. Conversely, suppose $\abs{A}$ is semisimple. Let $I_1,\ldots,I_{2m},I_{2m+1},\ldots,I_n$ be a complete set of pairwise non-isomorphic simple $\abs{A}$-modules, ordered so that $\pi_A(I_{2j}) \cong I_{2j-1}$ for $1 \leq j \leq m$, and $\pi_A(I_i) \cong I_i$ for $2m < i \leq n$. For $1 \leq i \leq n$, let $A^{I_i}$ be the sum of all minimal left ideals in $\abs{A}$ that are isomorphic to $I_i$ as left $\abs{A}$-modules. Then $\abs{A} = \bigoplus_{i=1}^n A^{I_i}$, and one has $\pi_A(A^{I_{2j}}) = A^{I_{2j-1}}$ for $1 \leq j \leq m$, and $\pi_A(A^{I_i}) = A^{I_i}$ for $2m < i \leq n$. This implies for $1 \leq j \leq m$ and $2m < i \leq n$ that $A^{I_{2j-1}} \oplus A^{I_{2j}}$ and $A^{I_i}$ are each sub\-super\-modules of the left regular representation of $A$. Given $1 \leq j \leq m$, fix a decomposition $A^{I_{2j}} = U_1 \oplus \cdots \oplus U_t$ of $A^{I_{2j}}$ into a direct sum of copies of $I_{2j}$. Then $A^{I_{2j-1}} \oplus A^{I_{2j}} = \bigoplus_{i=1}^t [\pi_A(U_i) \oplus U_i]$ is a direct sum decomposition of $A^{I_{2j-1}} \oplus A^{I_{2j}}$ into (Type Q) simple $A$-supermodules.

Now fix an integer $2m < i \leq n$, and set $I = I_i$. We will show that $A^I$ is a sum---hence a direct sum---of (Type M) simple $A$-supermodules. First, $A^I$ is a sum of minimal left ideals $U$ such that $U \cong I \cong \pi_A(I)$ as $\abs{A}$-modules, and for each of these ideals $U$ one has $U + \pi_A(U) \subseteq A^I$ because $\pi_A(A^I) = A^I$. Then $U + \pi_A(U)$ is an $A$-subsupermodule of $A^I$. Since $U$ is simple as an $\abs{A}$-module, one has either $U = \pi_A(U)$, in which case $U$ is a simple $A$-supermodule, or the sum $U + \pi_A(U)$ is direct. In the latter case, one can argue exactly as in the proof of \cref{lemma:self-associate} (but now, without reaching a contradiction) to show that $U + \pi_A(U)$ is a direct sum of two $A$-subsupermodules $U^+$ and $U^-$, each isomorphic as $A$-supermodules to $U$.
\end{proof}

Given a superalgebra $A$, one can check that $\Ann_A(\pi_A^*(M)) = \pi_A(\Ann_A(M))$ for each $\abs{A}$-module $M$. This implies that the Jacobson radical of $\abs{A}$ is closed under the parity map $\pi_A$, and hence is a superideal in $A$. Then the next lemma follows from \cref{lemma:ss-super-ordinary}.

\begin{lemma}[{\cite[Lemma 2.6]{Brundan:2002}}]
Let $A$ be a finite-dimensional superalgebra, and let $J = \rad(\abs{A})$ be the Jacobson radical of $\abs{A}$. Then $J$ is the unique smallest superideal of $A$ such that $A/J$ is a semisimple superalgebra. 
\end{lemma}

Finally, since each simple $A$-supermodule $M$ is a sum of simple $\abs{A}$-modules, one gets $\rad(\abs{A}) \subseteq \Ann_A(M)$, which implies that the superalgebras $A$ and $A/\rad(\abs{A})$ have the same simples. Then the next lemma follows by passing to the quotient $A/\rad(\abs{A})$, considering the left regular representations of $A$ and $\abs{A}$, and applying the Super Artin--Wedderburn Theorem.

\begin{lemma}[{\cite[Corollary 2.8]{Brundan:2002}}] \label{lemma:complete-irreducibles}
Let $A$ be a finite-dimensional superalgebra, and let $\set{V_1,\ldots,V_n}$ be a complete, irredundant set of simple $A$-supermodules (up to homogeneous isomorphism) such that $V_1,\ldots,V_m$ are of Type M and $V_{m+1},\ldots,V_n$ are of Type Q. For $m+1 \leq i \leq n$, write $V_i = V_i^+ \oplus V_i^-$ as a direct sum of simple $\abs{A}$-modules. Then
	\begin{equation} \label{eq:complete-A-irreducibles}
	\set{V_1,\ldots,V_m,V_{m+1}^{\pm},\ldots,V_n^{\pm}}
	\end{equation}
is a complete set of pairwise non-isomorphic simple $\abs{A}$-modules.
\end{lemma}

\subsection{Finite supergroups}

In this section, let $G$ be a finite group, and suppose $G$ contains a normal subgroup $H$ of index $2$. Let $\sgn : G \to G/H \cong \smallset{\pm 1}$ be the quotient homomorphism, considered also as a representation of $G$. Define a $\Z_2$-grading on $G$ by $G_{\zero} = H = \ker(\sgn)$ and $G_{\one} = G \backslash H$. This grading is multiplicative and it makes $G$ into a \emph{supergroup}. The $\Z_2$-grading on $G$ extends by linearity to a $\Z_2$-grading on the group algebra $\CG$, making $\CG$ into a superalgebra that we call the \emph{group superalgebra} of $G$. Since $\CG$ is semisimple as an ordinary algebra by Maschke's Theorem, then $\CG$ is semisimple as a superalgebra by \cref{lemma:ss-super-ordinary}.

Given a $\CH$-module $W$ and an element $t \in G_{\one}$, let ${}^t W = \smallset{ {}^t w : w \in W}$ be the \emph{conjugate} representation in which the action of an element $h \in H$ is defined by $h.{}^tw = {}^t [(tht^{-1}).w]$. Up to isomorphism, the conjugate representation does not depend on the particular choice of element in $G_{\one}$. We say that two $\CH$-modules $W$ and $W'$ are conjugate if $W' \cong {}^t W$ for some $t \in G_{\one}$. If $V$ is a $\CG$-module, we write $\Res_H^G(V)$ for the $\CH$-module obtained by restriction, and if $U$ is a $\CH$-module, we denote the induced $\CG$-module $\CG \otimes_{\CH} U$ by $\Ind_H^G(U)$.

\begin{proposition}[{\cite[Proposition 5.1]{Fulton:1991}}] \label{prop:index-2-restriction}
Let $V$ be a simple $\CG$-module. Then exactly one of the following holds:
	\begin{enumerate}
	\item $V \not\cong V \otimes \sgn$ as $\CG$-modules, $\Res_H^G(V)$ is simple and isomorphic to its conjugate, and $\Ind_H^G( \Res_H^G(V) ) \cong V \oplus (V \otimes \sgn)$. \label{item:V-not-iso-Vsgn}
	\item $V \cong V \otimes \sgn$ as $\CG$-modules, $\Res_H^G(V) = U' \oplus U''$ for $\CH$-submodules $U'$ and $U''$ that are simple and conjugate but not isomorphic, and $\Ind_H^G(U') \cong \Ind_H^G(U'') \cong V$.
	\end{enumerate}
Each simple $\CH$-module arises uniquely in this way, noting that in case (\ref{item:V-not-iso-Vsgn}) the simple $\CG$-modules modules $V$ and $V \otimes \sgn$ each determine the same $\CH$-module.
\end{proposition}

\begin{remark}
Given a $\CG$-(super)module $V$, it is immediate from the definitions that $V \otimes \sgn \cong \pi_{\CG}^*(V)$ as $\CG$-(super)modules. We emphasize however that the sign representation is \emph{not} a $\CG$-\emph{super}\-module, nor is the one-dimensional trivial $\CG$-module, though their direct sum is naturally a simple $G$-supermodule of Type Q.
\end{remark}

\subsection{Example: The group superalgebra of the dihedral group}

In this section, fix a positive integer $n \geq 3$ and let $D_n$ be the corresponding dihedral group of order $2n$. Write
	\[
	D_n = \subgrp{r,s : r^n = s^2 = (sr)^2 = 1} = \smallset{ 1, r, r^2, \ldots, r^{n-1}, s, sr, \ldots, sr^{n-1}}
	\]
and let $R_n = \smallset{ 1, r, r^2, \ldots, r^{n-1}}$ be the subgroup of rotations in $D_n$. Then $R_n$ is a normal subgroup of index $2$ in $D_n$, so $\CD_n$ is a superalgebra with $(\CD_n)_{\zero} = \CR_n$.

The irreducible complex representations of the group $D_n$ are given as follows:
	\begin{itemize}
	\item Let $\zeta = e^{2\pi i/n} \in \C$. Given an integer $k$, define $\rho_k: D_n \to GL_2(\C)$ by
		\begin{align*}
		\rho_k(r) &= \pmat{\zeta^k & 0 \\ 0 & \zeta^{-k}}, & \rho_k(s) &= \pmat{0 & 1 \\ 1 & 0}.
		\end{align*}
	These representations are irreducible and pairwise non-isomorphic provided that $1 \leq k < \frac{n}{2}$. Further, the representations $\rho_k$ and $\rho_{n-k}$ are isomorphic via conjugation by the matrix $\begin{psmallmatrix} 0 & 1 \\ 1 & 0 \end{psmallmatrix}$.

	\item The trivial representation $\rho_0: D_n \to GL_1(\C)$, defined by $\rho_0(r) = \pmat{1}$ and $\rho_0(s) = \pmat{1}$.

	\item The sign representation $\sgn : D_n \to GL_1(\C)$, defined by $\sgn(r) = \pmat{1}$ and $\sgn(s) = \pmat{-1}$.
	
	\item If $n$ is even, then there are two additional $1$-dimensional representations of $D_n$:
		\begin{itemize}
		\item $\rho_0^-: D_n \to GL_1(\C)$, defined by $\rho_0^-(r) = \pmat{-1}$ and $\rho_0^-(s) = \pmat{1}$.

		\item $\sgn^-: D_n \to GL_1(\C)$, defined by $\sgn^-(r) = \pmat{-1}$ and $\sgn^-(s) = \pmat{-1}$.
		\end{itemize}
	\end{itemize}

Now define subspaces of $\CD_n$ as follows:
	\begin{itemize}
	\item Given an integer $k$, let $\lambda = e^{2\pi i k/n} \in \C$, and let $V_k$ be the subspace of $\CD_n$ spanned by
		\[
		\sum_{i=0}^{n-1} \lambda^{-i} \cdot r^i \quad \text{and} \quad \sum_{j=0}^{n-1} \lambda^{-j} \cdot sr^j.
		\]

	\item Let $V_0$ be the subspace of $\CD_n$ spanned by
		\[
		\left( \sum_{i=0}^{n-1} r^i \right) + \left( \sum_{i=0}^{n-1} s r^i \right) \quad \text{and} \quad \left( \sum_{i=0}^{n-1} r^i \right) - \left( \sum_{i=0}^{n-1} s r^i \right).
		\]
	\end{itemize}

Then it is straightforward to check the following statements:
	\begin{itemize}
	\item For all integers $k$, $V_k$ is a subsuperspace of $\CD_n$, and
		\[
		\CD_n = \begin{cases}
		V_0 \oplus \left( \bigoplus_{1 \leq k < \frac{n}{2}} [V_k \oplus V_{n-k}] \right) \oplus V_{n/2} & \text{if $n$ is even,} \\
		V_0 \oplus \left( \bigoplus_{1 \leq k < \frac{n}{2}} [V_k \oplus V_{n-k}] \right) & \text{if $n$ is odd.}
		\end{cases}
		\]

	\item For each integer $1 \leq k < \frac{n}{2}$, $V_k$ and $V_{n-k}$ are Type M simple $\CD_n$-supermodules that afford the representations $\rho_k$ and $\rho_{n-k}$ of $D_{n}$, respectively.

	\item $V_0$ is a Type Q simple $\CD_n$-supermodule, whose restriction to $\abs{\CD_n}$ is $\rho_0 \oplus \sgn$.
	
	\item If $n$ is even, then $V_{n/2}$ is a Type Q simple $\CD_n$-supermodule, whose restriction to $\abs{\CD_n}$ is the direct sum of $\rho_0^-$ (spanned by $\sum_{i=0}^{n-1} (-1)^i r^i + \sum_{i=0}^{n-1} (-1)^i sr^i$) and $\sgn^-$ (spanned by $\sum_{i=0}^{n-1} (-1)^i r^i - \sum_{i=0}^{n-1} (-1)^i sr^i$).
	\end{itemize}

As a consequence of these observations and \cref{theorem:super-wedderburn}, we deduce the existence of a super\-algebra isomorphism
	\[
	\CD_n \cong \begin{cases}
	M(1|1)^{\oplus (\frac{n}{2} - 1)} \oplus Q(1)^{\oplus 2} & \text{if $n$ is even,} \\
	M(1|1)^{\oplus \floor{n/2}} \oplus Q(1) & \text{if $n$ is odd.}
	\end{cases}
	\]

\section{The symmetric group as a supergroup}\label{S:SnSupergroup}

In this section, fix an integer $n \geq 2$ and let $S_n$ be the symmetric group on $n$ letters. The sign representation $\sgn: S_n \to \set{\pm 1}$, $\sigma \mapsto (-1)^\sigma$, makes $S_n$ into a supergroup such that $(S_n)_{\zero} = A_n$, the alternating group on $n$ letters, and $(S_n)_{\one} = S_n \backslash A_n$ is the set of odd permutations. Then the group algebra $\CS_n$ becomes a superalgebra with $(\CS_n)_{\zero} = \CA_n$, the group algebra of $A_n$.

\subsection{The simple supermodules of the symmetric group} \label{subsec:group-algebra-Sn}

Write $\lambda \vdash n$ to denote that $\lambda$ is a partition of $n$, and let $\calP(n) = \set{ \lambda : \lambda \vdash n}$ be the set of all partitions of $n$. Given $\lambda \in \calP(n)$, write $\lambda'$ for the partition that is conjugate (or transpose) to $\lambda$, and let $\sim$ be the equivalence relation on $\calP(n)$ with equivalence classes $\left\{\left\{\lambda, \lambda' \right\} : \lambda \in \calP(n) \right\}$. 

\begin{definition}
Let $\Pbar(n)$ be any fixed set of representatives for the distinct equivalence classes in $\calP(n)$ under the relation $\sim$. Then $\Pbar(n)$ is a disjoint union of sets $E_n$ and $F_n$, where
	\[
	E_n = \smallset{ \lambda \in \Pbar(n) : \lambda \neq \lambda'} \quad \text{and} \quad F_n = \smallset{ \lambda \in \Pbar(n) : \lambda = \lambda'}.
	\]
\end{definition}

For $\lambda \vdash n$, let $S^\lambda$ be the corresponding Specht module. Then the set $\smallset{ S^\lambda : \lambda \vdash n}$ is a complete set of pairwise non-isomorphic simple $\CS_n$-modules. It is well-known that $S^\lambda \otimes \sgn \cong S^{\lambda'}$; see \cite[Theorems 4.12 and 6.7]{James:1978}. If $\lambda \neq \lambda'$, then \cref{prop:index-2-restriction} implies that $S^\lambda$ and $S^{\lambda'}$ are simple (and isomorphic) as $\CA_n$-modules, while for $\lambda = \lambda'$ one gets that $\Res_{A_n}^{S_n}(S^\lambda) = S^{\lambda^+} \oplus S^{\lambda^-}$ for two simple, conjugate, non-isomorphic $\CA_n$-modules $S^{\lambda^+}$ and $S^{\lambda^-}$. In particular, if $\tau \in S_n$ is any odd permutation, then multiplication by $\tau$ defines a linear isomorphism $S^{\lambda^+} \to S^{\lambda^-}$.

\begin{convention} \label{conv:(n)-in-En}
It will be convenient to assume that the partition $(n)$, corresponding to the trivial $\CS_n$-module $S^{(n)}$, is an element of $E_n$.
\end{convention}

Since $S^{\lambda'} \otimes \sgn \cong S^{\lambda}$, Schur's Lemma implies that $\Hom_{\CS_n}(S^\lambda,S^{\lambda'} \otimes \sgn) \cong \C$. For each $\lambda \vdash n$, choose a nonzero element $\phi^\lambda$ of this space, and interpret it as a linear isomorphism $\phi^\lambda : S^\lambda \to S^{\lambda'}$ such that
	\begin{equation} \label{eq:intertwinor}
	\phi^\lambda(\sigma \cdot v) = (-1)^{\sigma} \sigma \cdot \phi^\lambda(v) \quad \text{for all} \quad v \in S^\lambda \text{ and } \sigma \in S_n.
	\end{equation}
Then $\phi^{\lambda'} \circ \phi^\lambda \in \Hom_{\CS_n}(S^\lambda,S^\lambda) = \C \cdot \id_{S^\lambda}$. Rescaling our choice of $\phi^\lambda$ if necessary, we may assume that $\phi^{\lambda'} \circ \phi^\lambda = \id_{S^\lambda}$, and hence $\phi^{\lambda} \circ \phi^{\lambda'} = \id_{S^{\lambda'}}$, as well. Now for $\lambda = \lambda'$, we deduce that up to the rescaling $\phi^\lambda \mapsto -\phi^\lambda$, $\phi^\lambda$ is the unique self-inverse linear map satisfying \eqref{eq:intertwinor}, while for $\lambda \neq \lambda'$, we deduce that up to mutual rescalings of the form $(\phi^\lambda,\phi^{\lambda'}) \mapsto (c \cdot  \phi^\lambda, \frac{1}{c} \cdot \phi^{\lambda'})$, $\phi^\lambda$ and $\phi^{\lambda'}$ are the unique mutually-inverse linear maps each satisfying \eqref{eq:intertwinor}.

\begin{convention} \label{conv:associator-scaling}
For $\lambda = \lambda'$, we assume that the scaling of the map $\phi^\lambda: S^\lambda \to S^\lambda$ is fixed as in \cite[\S3]{Geetha:2018}. This ensures that whenever $(\mu,\lambda)$ is a self-conjugate cover in the sense of \cite[Definition 1]{Geetha:2018}, then $\phi^\lambda|_{S^\mu} = \phi^\mu$.
\end{convention}

Now for each symmetric partition $\lambda$, one has $(\phi^\lambda)^2 = \id_{S^\lambda}$, and hence $S^\lambda$ decomposes into $+1$ and $-1$ eigenspaces for $\phi^\lambda$. These eigenspaces are $A_n$-stable (because $\phi^\lambda$ is a $\CA_n$-homomorphism), and hence are $\CA_n$-submodules of $S^\lambda$. Moreover, neither eigenspace is equal to all of $S^\lambda$, since otherwise \eqref{eq:intertwinor} would imply for all $v \in S^\lambda$ that $\sigma.v = 0$ for all odd permutations (which is false). Combining these observations with those made two paragraphs ago, and using the uniqueness of isotypical components, one deduces that the $\pm 1$ eigenspaces of $\phi^\lambda$ are the simple $\CA_n$-constituents of $\Res_{A_n}^{S_n}(S^\lambda)$. We take $S^{\lambda^+}$ and $S^{\lambda^-}$ to be the $+1$ and $-1$ eigenspaces of $\phi^\lambda$, respectively.

\begin{lemma} \label{lemma:super-CSn-as-ordinary}
Let $n > 1$, and let $W$ be a simple $\CS_n$-supermodule.
	\begin{enumerate}
	\item If $W$ is of Type M, then $W \cong S^\lambda$ as a $\abs{\CS_n}$-module, for some symmetric partition $\lambda \vdash n$. Under this identification, the homogeneous subspaces of $W$ are $S^{\lambda^+}$ and $S^{\lambda^-}$.
	\item If $W$ is of Type Q, then $W \cong S^{\lambda} \oplus S^{\lambda'}$ as a $\abs{\CS_n}$-module, for some non-symmetric partition $\lambda \vdash n$. Under this identification, the homogeneous subspaces of $W$ are
		\[
		W_{\zero} = \smallset{ u+\phi^\lambda(u): u \in S^\lambda} \quad \text{and} \quad W_{\one} = \smallset{ u-\phi^\lambda(u): u \in S^\lambda}.
		\]
	\end{enumerate}
\end{lemma}

\begin{proof}
First suppose $W$ is of Type M. Then as a $\abs{\CS_n}$-module, $W \cong S^\lambda$ for some $\lambda \vdash n$. Since
	\[
	W = \pi_W(W) \cong \pi_{\CS_n}^*(W) \cong W \otimes \sgn \cong S^\lambda \otimes \sgn \cong S^{\lambda'}
	\]
as $\abs{\CS_n}$-modules, this implies that $\lambda = \lambda'$. Next, since the odd permutations in $S_n$ do not annihilate $S^\lambda$,  $W$ cannot be simply a purely even or a purely odd superspace. Then $W_{\zero}$ and $W_{\one}$ are nonzero $\CA_n$-submodules of $W$. Since $S^\lambda = S^{\lambda^+} \oplus S^{\lambda^-}$ as a $\CA_n$-module, the uniqueness of isotypical components implies that, as sets, $\smallset{S^{\lambda^+},S^{\lambda^-}} = \smallset{W_{\zero},W_{\one}}$.

Now suppose $W$ is of Type Q. Then by \cref{lemma:self-associate}, there exists $\lambda \vdash n$ such that, as a $\abs{\CS_n}$-module,
	\[
	W = S^{\lambda} \oplus \pi_W(S^\lambda) \cong S^{\lambda} \oplus \pi_{\CS_n}^*(S^\lambda) \cong S^{\lambda} \oplus S^{\lambda'},
	\]
and $S^{\lambda} \not\cong S^{\lambda'}$ as $\abs{\CS_n}$-modules. Then $\lambda \neq \lambda'$. Making the identification $\pi_W(S^\lambda) = S^{\lambda'}$, the parity map $\pi = \pi_W: W \to W$ restricts to mutually-inverse linear maps $\pi^\lambda : S^\lambda \to S^{\lambda'}$ and $\pi^{\lambda'}: S^{\lambda'} \to S^\lambda$ satisfying \eqref{eq:intertwinor}. Then by uniqueness (up to mutual rescaling) of $\phi^\lambda$ and $\phi^{\lambda'}$, we may assume that $\pi^{\lambda} = \phi^\lambda$ and $\pi^{\lambda'} = \phi^{\lambda'}$. Now the identification of $W_{\zero}$ and $W_{\one}$ follows from \cref{lemma:self-associate}.
\end{proof}

Note that dfferent choices for $\phi^{\lambda}$ would lead to different homogeneous subspaces of $S^{\lambda} \oplus S^{\lambda'}$ in the type $Q$ case above, but the description of $W$ would be the same up to an isomorphism.

\begin{proposition} \label{prop:Wlambda-decomposition}
Let $n > 1$.
\begin{enumerate}
\item \label{item:Wlambda-En} For each $\lambda \in E_n$, there exists a Type Q simple $\CS_n$-supermodule $W^\lambda$ such that $W^\lambda \cong S^\lambda \oplus S^{\lambda'}$ as a $\abs{\CS_n}$-module, with
	\[
	W^\lambda_{\zero} = \smallset{ u+\phi^\lambda(u) : u \in S^\lambda} \quad \text{and} \quad W^\lambda_{\one} = \smallset{ u-\phi^\lambda(u) : u \in S^\lambda}.
	\]
The $\abs{\CS_n}$-module decomposition $W^\lambda \cong S^{\lambda} \oplus S^{\lambda'}$ is canonical. We denote by $J^\lambda: W^{\lambda} \to W^{\lambda}$ the odd involution defined for $u \in S^\lambda$ by $J^\lambda(u \pm \phi^\lambda(u)) = u \mp \phi^\lambda(u)$.

\item \label{item:Wlambda-Fn} For each $\lambda \in F_n$, there exists a Type M simple $\CS_n$-supermodule such that $W^\lambda \cong S^\lambda$ as a $\abs{\CS_n}$-module, with $W^\lambda_{\zero} = S^{\lambda^+}$ and $W^\lambda_{\one} = S^{\lambda^-}$.
\end{enumerate}
The set $\smallset{ W^\lambda : \lambda \in \Pbar(n)}$ is a complete set of pairwise non-isomorphic simple $\CS_n$-super\-modules.
\end{proposition}

\begin{proof}
Using \cref{lemma:complete-irreducibles}, \cref{lemma:super-CSn-as-ordinary}, and the classification of the simple $\abs{\CS_n}$-modules, one deduces for each $\lambda \in \Pbar(n)$ that there exists a simple $\CS_n$-supermodule $W^\lambda$ with the given restriction to $\abs{\CS_n}$. In particular, if $\lambda \in E_n$, and if $W$ and $W'$ are Type Q simple $\CS_n$-super\-modules that are both isomorphic as $\abs{\CS_n}$-modules to $S^\lambda \oplus S^{\lambda'}$, then $W \cong W'$, so the notation $W^\lambda$ does not depend on the choice of representative for the equivalence class $\set{\lambda,\lambda'}$.

For $\lambda \in E_n$, the decomposition $W^\lambda \cong S^\lambda \oplus S^{\lambda'}$ is canonical by the uniqueness of isotypical com\-pon\-ents and the fact that $S^\lambda \not\cong S^{\lambda'}$ as $\abs{\CS_n}$-modules. For $\lambda \in F_n$, one can if needed replace $W^\lambda$ with its parity shift (to which $W^\lambda$ is odd-isomorphic) to ensure that $W^\lambda_{\zero} = S^{\lambda^+}$ and $W^\lambda_{\one} = S^{\lambda^-}$.
\end{proof}

If $G$ is a (super)group and if $V$ is a $\CG$-(super)module, then the linear dual $V^* = \Hom(V,\C)$ admits a $\CG$-(super)module structure, defined for $g \in G$, $\phi \in V^*$, and $v \in V$ by $(g.\phi)(v) = \phi(g^{-1}.v)$. We denote this group-theoretic module structure on $V^*$ by $V^{*,\Grp}$ when it is necessary to distinguish it from the Lie-algebraic module structure on $V^*$ that we consider later in \cref{prop:An-Lie-dual}.

The next remark considers the group-theoretic duals of the simple supermodules $W^\lambda$.

\begin{remark} \label{remark:self-dual}
For each $\lambda \in \calP(n)$, the Specht module $S^\lambda$ is self-dual \cite[Theorem 4.12]{James:1981}. From this and \cref{prop:Wlambda-decomposition}, it follows for each $\lambda \in \calP(n)$ that $(W^\lambda)^*$ is isomorphic (via an even isomorphism) to either $W^{\lambda}$ or $\Pi(W^\lambda)$. For $\lambda \in E_n$, one always has $(W^\lambda)^* \cong W^\lambda$ because $W^\lambda$ is even-isomorphic to $\Pi(W^\lambda)$, while for $\lambda \in F_n$ one has $(W^\lambda)^* \cong W^\lambda$ if and only if the $\CA_n$-modules $S^{\lambda^+}$ and $S^{\lambda^-}$ are each self-dual. The $\CA_n$-modules $S^{\lambda^+}$ and $S^{\lambda^-}$ are self-dual if and only if their complex characters are real-valued. By \cite[Proposition 5.3]{Fulton:1991}, this happens if and only if the number of squares above the diagonal in the Young diagram of $\lambda$ is even.
\end{remark}

The next result is an immediate consequence of \cite[Theorem 2.4.10]{James:1981}.

\begin{lemma} \label{lemma:lower-bound-dim}
Let $n \geq 2$, and let $\lambda \in \Pbar(n)$.
\begin{enumerate}
\item Suppose $\lambda \in E_n$. If $\lambda = (n)$ or $\lambda = (1^n)$, then $\dim(W^\lambda) = 2$. Otherwise, $\dim(W^\lambda) \geq 2n-2$.

\item Suppose $\lambda \in F_n$. If $n = 3$ and $\lambda = (2,1)$, or if $n = 4$ and $\lambda = (2,2)$, then $\dim(W^\lambda) = 2$. If $n = 5$ and $\lambda = (3,1,1)$, then $\dim(W^\lambda) = 6 = n+1$. Otherwise, $\dim(W^\lambda) \geq n+3$.
\end{enumerate}
\end{lemma}

Recall that we write $\End(V)$ for $\End_{\C}(V)$. Given a partition $\lambda \vdash n$ and a permutation $\sigma \in S_n$, let $S^\lambda(\sigma) \in \End(S^\lambda)$ and $W^\lambda(\sigma) \in \End(W^\lambda)$ denote the corresponding linear maps $u \mapsto \sigma . u$. For $\sigma \in A_n$, let $S^{\lambda^+}(\sigma) \in \End(S^{\lambda^+})$ and $S^{\lambda^-}(\sigma) \in \End(S^{\lambda^-})$ be defined similarly. By abuse of notation, we will also write $S^\lambda(\sigma)$, $W^\lambda(\sigma)$, etc., for the corresponding matrices when bases for the underlying modules are fixed, and we extend the notation $S^\lambda(\sigma)$ to arbitrary elements $\sigma \in \CS_n$ by linearity.

The next result is an immediate consequence of \cref{theorem:super-wedderburn} and \cref{prop:Wlambda-decomposition}.

\begin{corollary} \label{cor:CSn-as-superalgebra}
Let $n \geq 2$. The map $\CS_n \to \bigoplus_{\lambda \in \Pbar(n)} \End(W^\lambda)$, $\sigma \mapsto \bigoplus_{\lambda \in \Pbar(n)} W^\lambda(\sigma)$, induces a superalgebra isomorphism
		\begin{align*}
		\CS_n &\cong \Bigg[ \bigoplus_{\lambda \in E_n} Q\left( W^\lambda\right) \Bigg] \oplus \Bigg[ \bigoplus_{\lambda \in F_n} \End\left( W^\lambda\right) \Bigg] \cong \Bigg[ \bigoplus_{\lambda \in E_n} Q\left( f^\lambda\right) \Bigg] \oplus \Bigg[ \bigoplus_{\lambda \in F_n} M\left( \tfrac{1}{2}f^\lambda, \tfrac{1}{2}f^\lambda\right) \Bigg],
		\end{align*}
where $f^\lambda = \dim(S^\lambda)$.
\end{corollary}

Let $\lambda \in E_n$. For $u \in S^\lambda$, the expression $u \pm \phi^\lambda(u) \in W^\lambda$ is linear in $u$, and one has
	\begin{equation} \label{eq:sigma-on-homogeneous}
	\sigma . \big( u \pm \phi^\lambda(u) \big) = (\sigma . u) \pm (-1)^\sigma \phi^\lambda(\sigma . u)
	\end{equation}
for all $\sigma \in S_n$. Then making the identification $Q(W^\lambda) = Q(f^\lambda)$ via a choice of homo\-geneous basis as in \cref{lemma:self-associate}, the identity \eqref{eq:sigma-on-homogeneous} implies that
	\begin{equation} \label{eq:En-matrices}
	W^\lambda(\sigma) = \begin{cases}
	\left[ \begin{array}{c|c} S^\lambda(\sigma) & 0 \\ \hline 0 & S^\lambda(\sigma) \end{array} \right] & \text{if $\sigma$ is an even permutation,} \\
	\left[ \begin{array}{c|c} 0 & S^\lambda(\sigma) \\ \hline S^\lambda(\sigma) & 0 \end{array} \right] & \text{if $\sigma$ is an odd permutation.}
	\end{cases}
	\end{equation}

On the other hand, let $\lambda \in F_n$. Choose a basis $\set{u_1,\ldots,u_m}$ for $W^\lambda_{\zero} = S^{\lambda^+}$, and let $\tau \in S_n$ be a fixed odd permutation. Then $\set{\tau.u_1,\ldots,\tau.u_m}$ is a basis for $W^\lambda_{\one} = S^{\lambda^-}$. Now identifying $\End(W^\lambda)$ with $M(\frac{1}{2}f^\lambda,\frac{1}{2}f^\lambda)$ via this choice of homogeneous basis, one gets
	\begin{equation} \label{eq:Fn-matrices}
	W^\lambda(\sigma) = \begin{cases}
	\left[ \begin{array}{c|c} S^{\lambda^+}(\sigma) & 0 \\ \hline 0 & S^{\lambda^+}(\tau^{-1}\sigma \tau) \end{array} \right] & \text{if $\sigma$ is an even permutation,} \\
	\left[ \begin{array}{c|c} 0 & S^{\lambda^+}(\sigma \tau) \\ \hline S^{\lambda^+}(\tau^{-1}\sigma) & 0 \end{array} \right] & \text{if $\sigma$ is an odd permutation.}
	\end{cases}
	\end{equation}

\subsection{Weight space decompositions of Specht modules} \label{subsec:weight-space-Sn-modules}

Our main references for this section are \cite[\S2]{Kleshchev:2005} and \cite[\S3]{Ceccherini-Silberstein:2010}. Recall that the Jucys--Murphy elements $L_1,\ldots,L_n \in \CS_n$ are defined by $L_j = \sum_{1 \leq i < j} (i,j)$. In particular, $L_1 = 0$. The elements $L_1,\ldots,L_n$ generate a commutative, semisimple sub\-algebra of $\CS_n$. Since this subalgebra is semi\-simple, each finite-dimensional $\CS_n$-module $V$ decomposes into a direct sum of simultaneous eigenspaces for $L_1,\ldots,L_n$. Given $\alpha = (\alpha_1,\ldots,\alpha_n) \in \C^n$, the $\alpha$-weight space of $V$ is defined by
	\[
	V_{\alpha} = \set{ v \in V : L_i \cdot v = \alpha_i  v \text{ for all } 1 \leq i \leq n}.
	\]
Given $\alpha = (\alpha_1,\ldots,\alpha_n)$, we may write $\alpha(L_i) = \alpha_i$. The nonzero elements of $V_\alpha$ are called \emph{weight vectors}. If $V_\alpha \neq 0$, then we say that $\alpha$ is a weight of $V$. Let
	\[
	\calW(\lambda) = \smallset{ \alpha \in \C^n : \text{$\alpha$ is a weight of $S^\lambda$}},
	\]
and let $\calW(n) = \bigcup_{\lambda \vdash n} \calW(\lambda)$.

Given a partition $\lambda = (\lambda_1 \geq \lambda_2 \geq \cdots)$ of $n$, we draw the Young diagram of shape $\lambda$ via the so-called ``English'' convention (see \cite{wiki-Young-tableau}), as an array of boxes with $\lambda_i$ boxes in the $i$-th row, the rows of boxes lined up on the left. A box in the $i$-th row and $j$-th column of the diagram is said to have \emph{residue} $j-i$. A $\lambda$-tableau is a Young diagram of shape $\lambda$ in which the boxes have been filled in some order with the integers $1,2,\ldots,n$. A \emph{standard} $\lambda$-tableau is a $\lambda$-tableau in which the values of the integers increase from top to bottom along columns, and from left to right along rows.

Let $\T(\lambda)$ be the set of all standard $\lambda$-tableaux. The nonzero weight spaces of the simple $\CS_n$-module $S^\lambda$ are each one-dimen\-sional, spanned by vectors $v_T$ for $T \in \T(\lambda)$. Given $T \in \T(\lambda)$ and an integer $1 \leq i \leq n$, let $T_i$ be the box in $T$ that is occupied by $i$, and let $\res(T_i)$ be the residue of the box $T_i$. Then $v_T$ is of weight
	\[
	\alpha(T) := (\res(T_1),\ldots,\res(T_n)).
	\]
In particular, $\calW(\lambda) \subseteq \Z^n$. For example, if $n = 7$, $\lambda = (4,2,1)$, and
	\[
	T = \begin{ytableau}
	1 & 2 & 4 & 5 \\
	3 & 7 \\
	6
	\end{ytableau} \; ,
	\]
then $\alpha(T) = (0,1,-1,2,3,-2,0)$. This description implies that the union $\calW(n) = \bigcup_{\lambda \vdash n} \calW(\lambda)$ is disjoint, and for $\alpha \in \calW(n)$ one has $\alpha = -\alpha$ only if $n = 1$, which we have excluded by assumption.

We may denote a weight vector in $S^\lambda$ by $v_T$, for a standard $\lambda$-tableau $T$, or by $v_\alpha$, where $\alpha = \alpha(T)$ is the corresponding weight. Conversely, if  $\alpha \in \calW(\lambda)$ is specified, let $T(\alpha)$ be the corresponding standard $\lambda$-tableau. Then $v_{\alpha(T)} = v_T$ for all $T \in \T(\lambda)$, and $v_{T(\alpha)} = v_\alpha$ for all $\alpha \in \calW(\lambda)$. Given a (standard) $\lambda$-tableau $T$, let $T'$ be its transpose, which is then a (standard) $\lambda'$-tableau. Then for all $T \in \T(\lambda)$, one has $\alpha(T') = - \alpha(T)$, and for all $\alpha \in \calW(\lambda)$, one has	$T(-\alpha) = T(\alpha)'$.

\begin{proposition}[{\cite[Corollary 2.2.3]{Kleshchev:2005}}] \label{prop:si-on-weights}
Let $\alpha \in \calW(\lambda)$. Given $1 \leq i < n$, let $s_i$ be the transposition $(i,i+1) \in S_n$, and let $\beta = s_i.\alpha = (\alpha_1,\ldots,\alpha_{i-1},\alpha_{i+1},\alpha_i,\alpha_{i+2},\ldots,\alpha_n)$. Then:
	\begin{enumerate}
	\item $\alpha_i \neq \alpha_{i+1}$.

	\item \label{item:alphai-pm-alphai+1} If $\alpha_{i+1} = \alpha_i \pm 1$, then $s_i \cdot v_\alpha = \pm v_\alpha$ and $\beta \notin \calW(\lambda)$.

	\item Suppose $\alpha_{i+1} \neq \alpha_i \pm 1$, and let $c_i = (\alpha_{i+1}-\alpha_i)^{-1}$. Then $\beta \in \calW(\lambda)$ and $w_\beta := (s_i - c_i) \cdot v_\alpha$ is a nonzero scalar multiple of $v_\beta$; the elements $L_i$, $L_{i+1}$, and $s_i$ leave $S^\lambda_\alpha \oplus S^\lambda_\beta$ invariant; and they act in the basis $\set{v_\alpha,w_\beta}$ of $S^\lambda_\alpha \oplus S^\lambda_\beta$ via the matrices
		\[
		L_i = \begin{bmatrix} \alpha_i & 0 \\ 0 & \alpha_{i+1} \end{bmatrix}, \qquad
		L_{i+1} = \begin{bmatrix} \alpha_{i+1} & 0 \\ 0 & \alpha_i \end{bmatrix}, \qquad
		s_i = \begin{bmatrix} c_i & 1-c_i^2 \\ 1 & -c_i \end{bmatrix}.
		\]
	\end{enumerate}
\end{proposition}

\subsection{Weight space decompositions of simple supermodules} \label{subsec:weight-simple-super}

In this section we describe the actions of the odd operators $L_1,\ldots,L_n$ and the transpositions $s_1,\ldots,s_{n-1}$ on the simple $\CS_n$-super\-modules in terms of the weight vectors described in Section \ref{subsec:weight-space-Sn-modules}.

Given $\lambda \vdash n$ and $\alpha \in \calW(\lambda)$, it follows from the intertwining condition \eqref{eq:intertwinor} that the function $\phi^\lambda: S^\lambda \to S^{\lambda'}$ specified in Section \ref{subsec:group-algebra-Sn} defines a linear isomorphism $\phi^\lambda: S^\lambda_\alpha \xrightarrow{\cong} S^{\lambda'}_{-\alpha}$. We will assume that the spanning vectors $v_\alpha \in S^\lambda_\alpha$ and $v_{-\alpha} \in S^{\lambda'}_{-\alpha}$ are chosen so that
	\begin{equation} \label{eq:compatible-negative-weight-vector}
	v_{-\alpha} = \phi^\lambda(v_\alpha).
	\end{equation}
This can be done for all $\lambda \vdash n$ and $\alpha \in \calW(\lambda)$ because the union $\calW(n) = \bigcup_{\lambda \vdash n} \calW(\lambda)$ is disjoint, because $\alpha \neq -\alpha$ for all $\alpha \in \calW(n)$ by the assumption that $n > 1$, and because $\phi^{\lambda'} \circ \phi^\lambda = \id_{S^\lambda}$ and $\phi^\lambda \circ \phi^{\lambda'} = \id_{S^{\lambda'}}$. In terms of standard tableaux, one has $v_{T'} = \phi^\lambda(v_T)$ for all $T \in \T(\lambda)$.

The preceding discussion implies that the elements of $\calW(\lambda) \cup \calW(\lambda')$ occur in $\pm$ pairs. Let
	\[
	\Wbar(\lambda) = [\calW(\lambda) \cup \calW(\lambda')]/{\pm}
	\]
be the set of all such pairs. For $\lambda \in \Pbar(n) = E_n \cup F_n$, we will write $\pm \alpha$ to denote an element of $\Wbar(\lambda)$. This notation implicitly assumes a fixed choice for the `positive' element $\alpha$ of the pair $\pm \alpha$. If $\lambda \in E_n$, we will assume that $\alpha \in \calW(\lambda)$; if $\lambda \in F_n$, we will assume that $\alpha_n \geq 0$. This uniquely determines the choice of the positive element $\alpha$, except when $\lambda \in F_n$ and $\alpha_n = 0$. Now given $\lambda \in \Pbar(n)$, we will describe bases for $W^\lambda_{\zero}$ and $W^\lambda_{\one}$ that are indexed by $\Wbar(\lambda)$.

First let $\lambda \in E_n$, so that $W^\lambda \cong S^\lambda \oplus S^{\lambda'}$ as a $\abs{\CS_n}$-module. Given a pair $\pm \alpha \in \Wbar(\lambda)$, set
	\begin{align} \label{eq:En-valpha-pm}
	v_{\alpha}^+ &= \tfrac{1}{2}(v_\alpha + v_{-\alpha}) = \tfrac{1}{2}\big( v_\alpha + \phi^\lambda(v_\alpha) \big), & v_{\alpha}^- &= \tfrac{1}{2}(v_{\alpha} - v_{-\alpha}) = \tfrac{1}{2}\big( v_\alpha - \phi^\lambda(v_\alpha) \big).
	\end{align}
Then by \cref{lemma:self-associate}, the sets $\smallset{ v_\alpha^+ : \pm \alpha \in \Wbar(\lambda)}$ and $\smallset{ v_\alpha^- : \pm \alpha \in \Wbar(\lambda)}$ are bases for $W^\lambda_{\zero}$ and $W^\lambda_{\one}$, respectively. One has $v_\alpha = v_\alpha^+ + v_\alpha^-$ and $v_{-\alpha} = v_\alpha^+ - v_\alpha^-$.

Next let $\lambda \in F_n$, so that $W^\lambda \cong S^\lambda$ as a $\abs{\CS_n}$-module. Then $W^\lambda_{\zero} = S^{\lambda^+}$ and $W^\lambda_{\one} = S^{\lambda^-}$ are the $+1$ and $-1$ eigenspaces, respectively, for the function $\phi^\lambda: S^\lambda \to S^\lambda$. For each pair $\pm \alpha \in \Wbar(\lambda)$, write $v_\alpha = v_\alpha^+ + v_\alpha^-$, with $v_\alpha^+ \in S^{\lambda^+}$ and $v_\alpha^- \in S^{\lambda^-}$. Then $v_{-\alpha} = \phi^\lambda(v_\alpha) = v_{\alpha}^+ - v_{\alpha}^-$. Now
	\begin{align} \label{eq:Fn-valpha-pm}
	v_{\alpha}^+ &= \tfrac{1}{2}(v_\alpha + v_{-\alpha}) = \tfrac{1}{2}\big( v_\alpha + \phi^\lambda(v_\alpha) \big), & v_{\alpha}^- &= \tfrac{1}{2}(v_{\alpha} - v_{-\alpha}) = \tfrac{1}{2}\big( v_\alpha - \phi^\lambda(v_\alpha) \big),
	\end{align}
and the sets $\smallset{ v_\alpha^+ : \pm \alpha \in \Wbar(\lambda)}$ and $\smallset{ v_\alpha^- : \pm \alpha \in \Wbar(\lambda)}$ are bases for $W^\lambda_{\zero}$ and $W^\lambda_{\one}$, respectively.


With notation as above, one gets $L_i \cdot v_\alpha^+ = \alpha_i  v_{\alpha}^-$ and $L_i \cdot v_\alpha^- = \alpha_i  v_\alpha^+$ for each $1 \leq i \leq n$. For $\pm \alpha \in \Wbar(\lambda)$, set
	\[
	W^\lambda_{\pm \alpha} = \Span \set{ v_\alpha^+,v_\alpha^-} = \Span \set{v_\alpha,v_{-\alpha}}.
	\]
Then $W^\lambda = \bigoplus_{\pm \alpha \in \Wbar(\lambda)} W^\lambda_{\pm \alpha}$. We may refer to $W^\lambda_{\pm \alpha}$ as the $\pm \alpha$-weight space of $W^\lambda$.

The next result follows directly from \cref{prop:si-on-weights}.
	
\begin{proposition} \label{prop:si-on-weights-super}
Let $\alpha = (\alpha_1,\ldots,\alpha_n) \in \calW(\lambda)$. Let $1 \leq i < n$, and set $\beta = s_i.\alpha$.
	\begin{enumerate}
	\item \label{item:si-super-trivial-action} If $\alpha_{i+1} = \alpha_i \pm 1$, then the transposition $s_i$ leaves the superspace $W^\lambda_{\pm \alpha}$ invariant, and it acts in the homogeneous basis $\set{v_\alpha^+,v_\alpha^-}$ of $W^\lambda_{\pm \alpha}$ via the matrix
		\[
			\pm \begin{bmatrix} 0 & 1 \\ 1 & 0 \end{bmatrix},
		\]
	where the $\pm$ sign is the same as in \cref{prop:si-on-weights}(\ref{item:alphai-pm-alphai+1}).

	\item \label{item:si-super-nontrivial-action} Suppose $\alpha_{i+1} \neq \alpha_i \pm 1$. Let $c_i = (\alpha_{i+1}-\alpha_i)^{-1}$, and set
		\begin{align*}
		w_\beta &= (s_i - c_i) \cdot v_\alpha, \\
		w_{-\beta} &= (s_i+c_i) \cdot v_{-\alpha} = - \phi^\lambda(w_\beta), \\
		w_\beta^+ &= \tfrac{1}{2} ( w_\beta + \phi^\lambda(w_\beta) ) = \tfrac{1}{2}(w_\beta - w_{-\beta}), \\
		w_\beta^- &= \tfrac{1}{2} ( w_\beta - \phi^\lambda(w_\beta) ) = \tfrac{1}{2}(w_\beta+w_{-\beta}).
		\end{align*}
Then $\smallset{w_\beta^+,w_\beta^-}$ is a homogeneous basis for $W^\lambda_{\pm \beta}$, the elements $L_i$, $L_{i+1}$, and $s_i$ leave the space $W^\lambda_{\pm \alpha} \oplus W^\lambda_{\pm \beta}$ invariant, and they act in the homogeneous basis $\smallset{v_\alpha^+, w_\beta^+, v_\alpha^-,w_\beta^-}$ of this space via the following supermatrices:
		\begin{align*}
		L_i &= \left[ \begin{array}{cc|cc}
					0 & 0 & \alpha_i & 0 \\
					0 & 0 & 0 & \alpha_{i+1} \\
					\hline
					\alpha_i & 0 & 0 & 0 \\
					0 & \alpha_{i+1} & 0 & 0
					\end{array} \right], & 
		L_{i+1} &= \left[ \begin{array}{cc|cc}
					0 & 0 & \alpha_{i+1} & 0 \\
					0 & 0 & 0 & \alpha_{i} \\
					\hline
					\alpha_{i+1} & 0 & 0 & 0 \\
					0 & \alpha_{i} & 0 & 0
					\end{array} \right],
		\end{align*}
	and
		\[
		s_i =	\left[ \begin{array}{cc|cc}
					0 & 0 & c_i & 1-c_i^2 \\
					0 & 0 & 1 & -c_i \\
					\hline
					c_i & 1-c_i^2 & 0 & 0 \\
					1 & -c_i & 0 & 0
					\end{array} \right].
		\]
	\end{enumerate}
\end{proposition}

\subsection{Restriction of simple supermodules} \label{subsec:res-irred-super}

Given partitions $\lambda \vdash n$ and $\mu \vdash (n-1)$, write $\mu \prec \lambda$ if the Young diagram of $\mu$ is obtained by removing a box from the Young diagram of $\lambda$. In this case, let $\res(\lambda/\mu)$ denote the residue of the box that is removed from $\lambda$ to obtain $\mu$. Let $\res(\lambda)$ denote the sum of the residues of all the boxes in the Young diagram for $\lambda$.

Identify $S_{n-1}$ with the subgroup of $S_n$ of all permutations that leave the integer $n$ fixed. Then $\sigma \cdot L_n \cdot \sigma^{-1} = L_n$ for each $\sigma \in S_{n-1}$, and hence $L_n$ commutes (in the ordinary, non-super sense) with each element of $S_{n-1}$. This implies for each partition $\lambda \vdash n$ that $\Res_{S_{n-1}}^{S_n}(S^\lambda)$ decomposes into eigenspaces for the action of $L_n$. In fact, one has
	\begin{equation} \label{eq:S-lambda-restriction}
	\Res_{S_{n-1}}^{S_n}(S^\lambda) = \bigoplus_{\mu \prec \lambda} \Big[ \bigoplus_{\substack{\alpha \in \calW(\lambda) \\ \alpha_n = \res(\lambda/\mu)}} S^\lambda_\alpha \Big],
	\end{equation}
and the summand indexed by $\mu$ is isomorphic as a $\CS_{n-1}$-module to $S^\mu$.

Making the $\CS_{n-1}$-module identifications $S^\lambda = \bigoplus_{\mu \prec \lambda} S^\mu$ and $S^{\lambda'} = \bigoplus_{\mu' \prec \lambda'} S^{\mu'}$, and using the fact that $\Hom_{S_{n-1}}(S^\mu,S^\nu) = 0$ unless $\mu = \nu$, one can show that the functions $\phi^\lambda$ and $\phi^{\lambda'}$ must restrict for each $\mu \prec \lambda$ to linear isomorphisms $S^\mu \to S^{\mu'}$ and $S^{\mu'} \to S^\mu$ that satisfy the relation \eqref{eq:intertwinor} for all $\sigma \in S_{n-1}$, and whose composites are the respective identity functions. Then we may assume that $\phi^\lambda|_{S^\mu} = \phi^\mu$ and $\phi^{\lambda'}|_{S^{\mu'}} = \phi^{\mu'}$; cf.\ \cref{conv:associator-scaling}.
	
Now let $\lambda \in \Pbar(n) = E_n \cup F_n$. As a superspace, one has
	\begin{equation} \label{eq:Wlambda-integer-eigenspaces}
	W^\lambda = \bigoplus_{k \in \Z} W_k^\lambda, \quad \text{where} \quad W_k^\lambda = \bigoplus_{\substack{\pm \alpha \in \Wbar(\lambda) \\ \alpha_n = k}} W^\lambda_{\pm \alpha}.
	\end{equation}
By our conventions for the choice of the `positive' weight $\alpha$ from each pair $\pm \alpha \in \Wbar(\lambda)$, if $\lambda \in F_n$ and $W_k^\lambda \neq 0$, then $k \geq 0$. In general, if $W_k^\lambda \neq 0$, then there exists a unique partition $\mu \vdash (n-1)$ such that $\mu \prec \lambda$ and $\res(\lambda/\mu) = k$. Specifically, $\mu$ is the partition obtained by removing a box of residue $k$ from the outer edge of the Young diagram of $\lambda$. Indeed, a box of residue $k$ can be removed from the outer edge of the Young diagram of $\lambda$ to produce a new partition $\mu$ if and only if there exists a weight $\alpha \in \calW(\lambda)$ with $\alpha_n = k$, and for any given $k$ there is at most one removable box of residue $k$ in the Young diagram of $\lambda$. For any $\lambda \vdash n$, the boxes in the Young diagram of $\lambda$ have residues bounded by $\pm (n-1)$, so in \eqref{eq:Wlambda-integer-eigenspaces} one has $W_k^\lambda \neq 0$ only if $\abs{k} < n$.

Since $L_n$ commutes with $S_{n-1}$, it follows that $W_k^\lambda$ is a $\CS_{n-1}$-subsuper\-module of $W^\lambda$.

\begin{proposition} \label{prop:Wklambda-restriction}
Let $\lambda \in \Pbar(n)$, let $k \in \Z$ such that $W_k^\lambda \neq 0$, and let $\mu \vdash (n-1)$ be the unique partition such that $\mu \prec \lambda$ and $\res(\lambda/\mu) = k$. Then as a $\CS_{n-1}$-supermodule,
	\[
	W_k^\lambda \cong \begin{cases}
	W^\mu \oplus \Pi(W^\mu) & \text{if $\lambda \in E_n$ and $\mu = \mu'$,} \\
	W^\mu & \text{otherwise.}
	\end{cases}
	\]
\end{proposition}

\begin{proof}
First suppose $\lambda \in E_n$. Then
	\[
	W_k^\lambda = \bigoplus_{\substack{\alpha \in \calW(\lambda) \\ \alpha_n = k}} \left[ S^\lambda_\alpha \oplus S^{\lambda'}_{-\alpha} \right] = \Big[ \bigoplus_{\substack{\alpha \in \calW(\lambda) \\ \alpha_n = \res(\lambda/\mu)}} S^\lambda_\alpha \Big] \oplus \Big[ \bigoplus_{\substack{\alpha \in \calW(\lambda') \\ \alpha_n = \res(\lambda'/\mu')}} S^{\lambda'}_\alpha \Big].
	\]
By \eqref{eq:S-lambda-restriction}, this is isomorphic as a $\abs{\CS_{n-1}}$-module to $S^\mu \oplus S^{\mu'}$. For $\mu \neq \mu'$, this implies that $W_k^\lambda \cong W^\mu$ as a $\CS_{n-1}$-supermodule, so suppose that $\mu = \mu'$. Making the $\abs{\CS_{n-1}}$-module identification
	\[
	S^\mu = \bigoplus_{\substack{\alpha \in \calW(\lambda) \\ \alpha_n = \res(\lambda/\mu)}} S^\lambda_\alpha,
	\]
one sees that $W_k^\lambda$ decomposes into the direct sum of two $\CS_{n-1}$-supermodules,
	\begin{align} 
	W^\mu &\cong \smallset{ u + \phi^\lambda(u) : u \in S^{\mu^+}} \oplus \smallset{w - \phi^\lambda(w) : w \in S^{\mu^-}}, \quad \text{and} \label{eq:Wmu-identification} \\
	\Pi(W^\mu) &\cong \smallset{ u - \phi^\lambda(u) : u \in S^{\mu^+}} \oplus \smallset{w + \phi^\lambda(w) : w \in S^{\mu^-}}. \label{eq:Pi(Wmu)-identification}
	\end{align}
On the right-hand side of the isomorphism in \eqref{eq:Wmu-identification}, the first (resp.\ second) summand is contained in the even (resp.\ odd) subspace of $W_k^\lambda$, while on the right-hand side of \eqref{eq:Pi(Wmu)-identification}, the first (resp.\ second) summand is contained in the odd (resp.\ even) subspace of $W_k^\lambda$. From these identifications, one sees that $W^\mu$ and $\Pi(W^\mu)$ are interchanged by the odd involution $J^\lambda : W^\lambda \to W^\lambda$.

Now suppose $\lambda \in F_n$. If $k > 0$, then the partition $\mu$ is non-symmetric, $\res(\lambda/\mu') = -k$, and
	\[
	W_k^\lambda = \Big[ \bigoplus_{\substack{\alpha \in \calW(\lambda) \\ \alpha_n = \res(\lambda/\mu)}} S^\lambda_\alpha \Big] \oplus \Big[ \bigoplus_{\substack{\alpha \in \calW(\lambda) \\ \alpha_n = \res(\lambda/\mu')}} S^{\lambda}_\alpha \Big] \cong S^\mu \oplus S^{\mu'}
	\]
as $\abs{\CS_{n-1}}$-modules. This implies that $W_k^\lambda \cong W^\mu$ as a $\CS_{n-1}$-supermodule. On the other hand, if $k = 0$, then $\mu$ is symmetric, and
	\[
	W_k^\lambda = \Big[ \bigoplus_{\substack{\alpha \in \calW(\lambda) \\ \alpha_n = \res(\lambda/\mu)}} S^\lambda_\alpha \Big] \cong S^\mu
	\]
as a $\abs{\CS_{n-1}}$-module. Since $\phi^\lambda: S^\lambda \to S^\lambda$ restricts to $\phi^\mu: S^\mu \to S^\mu$ via this identification, one deduces that the $+1$-eigenspace of $\phi^\mu$ is contained in the $+1$-eigenspace of $\phi^\lambda$. Then $S^{\mu^+}$ is concentrated in even superdegree, so $W_k^\lambda \cong W^\mu$ as a $\CS_{n-1}$-supermodule.
\end{proof}

Let $W^{\lambda} \da{\CS_{n-1}}$ denote the restriction of $W^\lambda$ to the subalgebra $\CS_{n-1}$ of $\CS_n$.

\begin{corollary} \label{cor:Wlambda-restriction}
Let $\lambda \in \Pbar(n)$. Then
	\[
	W^\lambda \da{\CS_{n-1}} \cong \begin{cases}
	\displaystyle \Bigg[ \bigoplus_{\substack{\mu \prec \lambda \\ \mu \neq \mu'}} W^\mu \Bigg] \oplus \Bigg[ \bigoplus_{\substack{\mu \prec \lambda \\ \mu = \mu'}} W^\mu \oplus \Pi(W^\mu) \Bigg] & \text{if $\lambda \in E_n$,} \\
	\displaystyle \bigoplus_{\substack{\mu \prec \lambda \\ \res(\lambda/\mu) \geq 0}} W^\mu & \text{if $\lambda \in F_n$.}
	\end{cases}
	\]
\end{corollary}

\begin{remark} \label{rem:multiplicity-free}
The corollary implies that, if one allows only even supermodule homomorphisms (so that a supermodule and its parity shift are not necessarily isomorphic), then the restriction $W^\lambda \da{\CS_{n-1}}$ is multiplicity free, just as in the classical (non-super) situation for Specht modules. If, on the other hand, one allows odd isomorphisms as well (so that a supermodule and its parity shift are \emph{odd} isomorphic), then the restriction $W^\lambda \da{\CS_{n-1}}$ is multiplicity free if $\lambda \in F_n$, but may have a  (unique) repeated composition factor if $\lambda \in E_n$.
\end{remark}

\begin{remark} \label{remark:An-restriction}
The following restriction formulas for the alternating groups can be deduced by taking homogeneous subspaces in \cref{cor:Wlambda-restriction}; see also \cite[\S6]{Ruff:2008} or \cite[\S4]{Geetha:2018}:
	\begin{itemize}
	\item If $\lambda \in E_n$, then $\Res^{A_n}_{A_{n-1}}(S^\lambda) \cong \left[ \bigoplus_{\substack{\mu \prec \lambda \\ \mu \neq \mu'}} S^\mu \right] \oplus \Bigg[ \bigoplus_{\substack{\mu \prec \lambda \\ \mu = \mu'}} S^{\mu^+} \oplus S^{\mu^-} \Bigg]$

	\item If $\lambda \in F_n$, then $\Res^{A_n}_{A_{n-1}}(S^{\lambda^{\pm}}) \cong \left[ \bigoplus_{\substack{\mu \prec \lambda \\ \res(\lambda/\mu) > 0}} S^\mu \right] \oplus \Bigg[ \bigoplus_{\substack{\mu \prec \lambda \\ \res(\lambda/\mu) = 0}} S^{\mu^{\pm}} \Bigg]$
	\end{itemize}
\end{remark}

\section{The Lie superalgebra generated by transpositions}\label{S:LSAgeneratedbytranspositions}

\subsection{The setup} \label{subsec:setup}

Recall the associative superalgebras defined in \cref{ex:M(m|n)} and \cref{ex:Q(n)}. Given a vector superspace $V \cong \C^{m|n}$, write $\gl(V)$ and $\glmn$ for the sets $\End(V)$ and $M(m|n)$, respectively, considered as Lie super\-algebras via the super commutator
	\[
	[x,y] = xy - (-1)^{\ol{x} \cdot \ol{y}}yx.
	\]
If $V$ is a vector superspace equipped with an odd involution $J : V \to V$ (so in particular, the even and odd subspaces of $V$ must be of the same dimension), write $\fq(V)$ and $\fq(n)$ for the sets $Q(V)$ and $Q(n)$, respectively, considered as Lie super\-algebras via the super commutator. For an arbitrary Lie superalgebra $\g$, we denote its derived subalgebra $[\g,\g]$ by $\fD(\g)$. Then
	\[
	\fD(\glmn) = \fsl(m|n) := \Set{ \left[	\begin{array}{c|c}
								A & B \\
								\hline
								C & D
								\end{array} \right] \in \glmn : \tr(A) - \tr(D) = 0}.
	\]

Let $V$ be a vector superspace equipped with an odd involution $J : V \to V$, let $\theta \in \fq(V)$, and let $\theta = \theta_{\zero} + \theta_{\one}$ be the decomposition of $\theta$ into its even and odd components. Then $J \circ \theta_{\one} = \theta_{\one} \circ J$ restricts to an even linear map $(J \circ \theta_{\one})|_{\Vzero} : \Vzero \to \Vzero$. Identifying $\Vzero$ and $\Vone$ via $J$, this is equal to the even linear map $(\theta_{\one} \circ J)|_{\Vone} : \Vone \to \Vone$. Now define the \emph{odd trace} of $\theta$, denoted $\otr(\theta)$, by
	\begin{equation} \label{eq:odd-trace}
	\otr(\theta) = \tr\big( (J \circ \theta_{\one})|_{\Vzero} \big) = \tr\big( (\theta_{\one} \circ J)|_{\Vone} \big),
	\end{equation}
and define the subsuperspace $\sq(V) \subseteq \fq(V)$ by
	\[
	\sq(V) = \Set{ \theta \in \fq(V) : \otr(\theta) = 0 }.
	\]
Then one can show that $\fD(\fq(V)) = \sq(V)$. Fixing a basis for $V$ as in \cref{ex:Q(n)}, one has
	\begin{equation} \label{eq:sq(n)-supermatrices}
	\fD(\fq(n)) = \sq(n) := \Set{ \left[	\begin{array}{c|c}
								A & B \\
								\hline
								B & A
								\end{array} \right] \in \fq(n) : \tr(B) = 0 }.
	\end{equation}
	
\begin{lemma} \ \label{lem:generated-by-odd}
	\begin{enumerate}
	\item If $m \geq 2$, then $\fsl(m|m)$ is generated as a Lie superalgebra by $\fsl(m|m)_{\one}$.
	\item If $m \geq 3$, then $\sq(m)$ is generated as a Lie superalgebra by $\sq(m)_{\one}$.
	\end{enumerate}
\end{lemma}

\begin{proof}
It is an exercise to show that various matrix units (or sums of two matrix units) spanning the even parts of the Lie super\-algebras can be obtained as Lie brackets of odd elements.
\end{proof}

Given a supergroup $G$, write $\Lie(\CG)$ for the group algebra $\CG$ considered as a Lie superalgebra via the super commutator $[x,y] = xy-(-1)^{\ol{x} \cdot \ol{y}}yx$, and set $\fD(\CG) = \fD(\Lie(\CG))$. \cref{cor:CSn-as-superalgebra} then gives the Lie superalgebra isomorphism
		\begin{equation} \label{eq:CSn-Lie-superalgebra-iso}
		\Lie(\CS_n) \cong \Bigg[ \bigoplus_{\lambda \in E_n} \fq(W^\lambda) \Bigg] \oplus \Bigg[ \bigoplus_{\lambda \in F_n} \gl(W^\lambda) \Bigg] \cong \Bigg[ \bigoplus_{\lambda \in E_n} \fq(f^\lambda) \Bigg] \oplus \Bigg[ \bigoplus_{\lambda \in F_n} \gl(\tfrac{1}{2}f^\lambda, \tfrac{1}{2}f^\lambda) \Bigg],
		\end{equation}
where $f^\lambda = \dim(S^\lambda)$. Taking derived subalgebras, one has
		\begin{equation} \label{eq:D(CSn)}
		\fD(\CS_n) \cong \Bigg[ \bigoplus_{\lambda \in E_n} \sq(W^\lambda) \Bigg] \oplus \Bigg[ \bigoplus_{\lambda \in F_n} \fsl(W^\lambda) \Bigg] \cong \Bigg[ \bigoplus_{\lambda \in E_n} \sq(f^\lambda) \Bigg] \oplus \Bigg[ \bigoplus_{\lambda \in F_n} \fsl(\tfrac{1}{2}f^\lambda, \tfrac{1}{2}f^\lambda) \Bigg].
		\end{equation}
From this one sees that
	\begin{align*}
	\dim(\fD(\CS_n)_{\zero}) &= \dim((\CS_n)_{\zero}) - \abs{F_n} = \tfrac{n!}{2} - \abs{F_n}, \quad \text{and} \\
	\dim(\fD(\CS_n)_{\one}) &= \dim((\CS_n)_{\one}) - \abs{E_n} = \tfrac{n!}{2} - \abs{E_n}.
	\end{align*}
In total, $\dim(\fD(\CS_n)) = n! - \abs{E_n \cup F_n}$.

Let $T_n = \sum_{j=1}^n L_j = \sum_{j=1}^n \sum_{i=1}^{j-1} (i,j)$, the sum in $\CS_n$ of all transpositions $\tau \in S_n$. Let $\lambda \vdash n$. Since all trans\-positions are conjugate in $S_n$, the trace of the map $S^\lambda(\tau): S^\lambda \to S^\lambda$ is independent of $\tau$. This implies for each transposition $\tau$ that $\tr( S^\lambda(T_n) ) = \binom{n}{2} \cdot \tr ( S^\lambda(\tau) )$, and hence $\tau - \frac{2}{n(n-1)} \cdot T_n$ is a traceless operator on $S^\lambda$. Then \eqref{eq:En-matrices} implies for $\lambda \in E_n$ that
	\[ \textstyle
	W^\lambda \left( \tau - \frac{2}{n(n-1)} \cdot T_n \right) \in \sq(W^\lambda) = \fD(\fq(W^\lambda)),
	\]
while for $\lambda \in F_n$ one gets
	\[ \textstyle
	W^\lambda \left( \tau - \frac{2}{n(n-1)} \cdot T_n \right) \in \fsl(W^\lambda) = \fD(\gl(W^\lambda)),
	\]
because for any finite-dimensional superspace $V$ one has $\gl(V)_{\one} = \fsl(V)_{\one} = \fD(\gl(V))_{\one}$. Combining these observations, one gets
	\begin{equation} \label{eq:generators-in-derived}
	\textstyle \set{ \tau - \frac{2}{n(n-1)} \cdot T_n : \tau \text{ is a transposition in $S_n$}} \subseteq \fD(\CS_n).
	\end{equation}
On the other hand, for any partition $\lambda \vdash n$, $T_n$ acts on $S^\lambda$ as scalar multiplication by $\res(\lambda)$, the sum of the residues of the boxes in the Young diagram of $\lambda$. For $n > 1$, there are non-symmetric partitions for which this scalar is nonzero, so this implies by \eqref{eq:En-matrices} that $T_n \notin \fD(\CS_n)$, and hence $\tau \notin \fD(\CS_n)$, as well, for each transposition $\tau$. Since $T_n$ acts on $S^\lambda$ as scalar multiplication by $\res(\lambda)$, it follows that
	\begin{equation} \label{eq:Wlambda(Tn)}
	W^\lambda(T_n) = \begin{cases}
	\res(\lambda) \cdot J^\lambda & \text{if $\lambda \in E_n$,} \\
	0 & \text{if $\lambda \in F_n$,}
	\end{cases}
	\end{equation}
where $J^\lambda: W^\lambda \to W^\lambda$ is the odd involution defined in \cref{prop:Wlambda-decomposition}\eqref{item:Wlambda-En}.

\begin{remark}
If $\res(\lambda) \neq 0$, then $\lambda \neq \lambda'$. The converse of this statement is false. For example, if $\lambda = (5,5,5,3,1,1)$, then $\lambda \neq \lambda'$ but $\res(\lambda) = 0$.
\end{remark}

\begin{definition}
Let $\g_n \subseteq \CS_n$ be the Lie superalgebra generated by all transpositions in $S_n$.
\end{definition}

Evidently, $T_n \in \g_n$. Then \eqref{eq:generators-in-derived} implies that
	\begin{equation} \label{eq:gn-subset-derived-plus-Tn}
	\g_n \subseteq \fD(\CS_n) + \C \cdot T_n.
	\end{equation}
Our goal by the end of the paper is to show that \eqref{eq:gn-subset-derived-plus-Tn} is an equality for all $n \geq 2$.

\begin{lemma} \label{lemma:gn-equality}
If $n \in \set{2,3,4,5}$, then $\g_n = \fD(\CS_n) + \C \cdot T_n$.
\end{lemma}

\begin{proof}
Since $T_n \notin \fD(\CS_n)$, the sum $\fD(\CS_n) + \C \cdot T_n$ is direct, and hence
	\[
	\dim(\fD(\CS_n) + \C \cdot T_n) = \dim(\fD(\CS_n))+1 = n! - \abs{E_n \cup F_n} + 1.
	\]
For $n \in \set{2,3,4,5}$, we have verified that $\dim(\g_n) \geq n! - \abs{E_n \cup F_n} + 1$, and hence \eqref{eq:gn-subset-derived-plus-Tn} is an equality, via calculations in GAP \cite{GAP4}.
\end{proof}

\begin{remark}
It is straightforward, if somewhat tedious, to check by hand for $n \in \set{2,3,4}$ that $\dim(\g_n) \geq n! - \abs{E_n \cup F_n} + 1$. Later in Section \ref{subsec:proof-En}, we will find it convenient to assume that \eqref{eq:gn-subset-derived-plus-Tn} is an equality for $n=5$, as well, to help avoid certain annoying special cases. In fact, we have verified that \cref{lemma:gn-equality} is also true for $n=6$ and $n=7$ using GAP, but there is nothing to be gained in our induction argument by taking these cases for granted.
\end{remark}

\begin{lemma} \label{lemma:lie-algebra-centers}
Let $n \geq 2$. Set $\g = \g_n$.
	\begin{enumerate}
	\item $Z(\CS_n) = Z(\CS_n)_{\zero}$.
	\item $Z(\g) \subseteq Z(\CS_n)$. In particular, $Z(\g) \subseteq \gzero$.
	\item If $n \geq 5$, then $Z(\gzero) = \gzero \cap Z(\CA_n)$, and the projection map $p: \CA_n \to Z(\CA_n)$,
		\[
		p(z) = \frac{2}{n!}  \sum_{\sigma \in A_n} \sigma z \sigma^{-1},
		\]
	restricts to a projection map $p: \gzero \to Z(\gzero)$. For this map, one has $p(\fD(\gzero)) \subseteq \fD(\gzero)$.
	\end{enumerate}
\end{lemma}

\begin{proof}
Under the superalgebra isomorphism of \cref{cor:CSn-as-superalgebra}, one sees that the only homo\-gen\-eous elements of $\CS_n$ that commute (in the super sense) with all other elements of $\CS_n$ correspond to linear combinations of the identity elements from the various matrix factors. In particular, $Z(\CS_n)$ is a purely even superspace.

Next, for each $z \in \CS_n$, the map $\ad_z: x \mapsto [z,x] = zx - (-1)^{\ol{z} \cdot \ol{x}} xz$ is a superalgebra derivation on $\CS_n$. If $z \in Z(\g)$, then $\ad_z(x) = 0$ for each transposition $x$, since those elements generate $\g$ as a Lie superalgebra. But the transpositions also generate $\CS_n$ as an associative superalgebra, so this implies that $\ad_z: \CS_n \to \CS_n$ is the zero map, and hence $z \in Z(\CS_n)$.

Now suppose $n \geq 5$. In this case, it is well-known that $A_n$ is generated as a group by the set
	\begin{equation} \label{eq:An-generators}
	\set{ (i,j)(k,\ell): \text{$i,j,k,\ell$ distinct}}
	\end{equation}
of all products of two disjoint transpositions. These are all elements of $\gzero$ because
	\begin{equation} \label{eq:An-generator}
	[(i,j),(k,\ell)] = (i,j)(k,\ell) + (k,\ell)(i,j) = 2  (i,j)(k,\ell)
	\end{equation}
whenever $i,j,k,\ell$ are distinct. Then reasoning as in the previous paragraph, it follows for $z \in \gzero$ that $z \in Z(\gzero)$ if and only if $z \in Z(\CA_n)$. Finally, since the set of transpositions in $S_n$ is closed under conjugation by arbitrary elements of $S_n$, it follows that $\g$ is closed under conjugation. Conjugation is an even linear map, so $\gzero$ is also closed under conjugation. Then the projection map must send elements of $\gzero$ to elements of $\gzero \cap Z(\CA_n) = Z(\gzero)$. Since $\gzero$ is closed under conjugation, it follows that $\fD(\gzero) = [\gzero,\gzero]$ is also closed under conjugation, and hence $p(\fD(\gzero)) \subseteq \fD(\gzero)$.
\end{proof}

\subsection{\texorpdfstring{Image of $\g_n$ in $\End(W^\lambda)$}{Image of gn in End(Wlambda)}} \label{S:Imageofgn}

Given $\lambda \in \Pbar(n)$, let $W^\lambda(\g_n)$ denote the image of $\g_n$ under the super\-module structure map $\CS_n \to \End(W^\lambda)$. Our goal in Sections \ref{subsec:proof-Fn} and \ref{subsec:proof-En} is to establish \cref{theorem:Wlambda(gn)}, stated below.

As described in the introduction, we prove the results in Sections \ref{S:Imageofgn}--\ref{subsec:gn-structure} by induction on $n$. First, for the base case of induction, observe that \cref{theorem:finaltheorem} is true for $n \in \set{2,3,4,5}$, by \cref{lemma:gn-equality}. This implies that \cref{theorem:Wlambda(gn)} is true for $n \in \set{2,3,4,5}$ by \eqref{eq:D(CSn)} and \eqref{eq:Wlambda(Tn)}. Hence \cref{cor:g0-reductive} and \cref{cor:Wlambda(D)} (whose proofs for a given value of $n$ rely only on the statement of \cref{theorem:Wlambda(gn)} for the same value of $n$) are true for $n$ in this range as well. In the case $n=5$, one can then work sequentially through Sections \ref{subsec:detecting-isos} and \ref{subsec:gn-structure} to deduce that all subsequent results in the paper leading up to \cref{theorem:finaltheorem} are also true for $n=5$. Now for the general inductive step of this  argument we make the following assumptions:
\begin{itemize}
	\item $n \geq 6$, and
	
	\item all results in Sections \ref{S:Imageofgn}--\ref{subsec:gn-structure} are true as stated for the value $n-1$.
\end{itemize}
The inductive step is then completed by working sequentially through Sections \ref{S:Imageofgn}--\ref{subsec:gn-structure}, starting with \cref{theorem:Wlambda(gn)}, to establish that each result is true as stated for the given value $n$.

\begin{theorem} \label{theorem:Wlambda(gn)}
Let $n \geq 2$, and let $\lambda \in \Pbar(n)$. Then
	\begin{equation} \label{eq:Wlambda(gn)}
	W^\lambda(\g_n) = \begin{cases} 
	\sq(W^\lambda) + \C \cdot (\res(\lambda) \cdot J^\lambda) & \text{if $\lambda \in E_n$,} \\
	\fsl(W^\lambda) & \text{if $\lambda \in F_n$.}
	\end{cases}
	\end{equation}
\end{theorem}

The ``$\subseteq$'' direction of \eqref{eq:Wlambda(gn)} follows from \eqref{eq:gn-subset-derived-plus-Tn}, \eqref{eq:D(CSn)}, and \eqref{eq:Wlambda(Tn)}. If $\lambda \in E_n$, then \eqref{eq:Wlambda(Tn)} also implies that $\res(\lambda) \cdot  J^\lambda \in W^\lambda(\g_n)$. Next, if $\tau \in S_n$ is any transposition, then
	\[
	\id_{W^\lambda} = W^\lambda(1_{\CS_n}) = W^\lambda\left( \tfrac{1}{2}[\tau,\tau] \right) \in W^\lambda(\g_n).
	\]
For $\lambda = (n)$, one has $\sq(W^\lambda) = \C \cdot \id_{W^\lambda}$, so the theorem is true in this case. For all other $\lambda \in \Pbar(n)$, \cref{lemma:lower-bound-dim} implies (by the assumption $n \geq 6$, and the fact that all $W^\lambda$ are even-dimensional) that $\dim(W^\lambda) \geq 10$. Then to finish proving the ``$\supseteq$'' direction of \eqref{theorem:Wlambda(gn)}, it suffices by \cref{lem:generated-by-odd} to show that
	\begin{equation} \label{eq:sufficient-inclusion}
	W^\lambda(\g_n) \supseteq \begin{cases}
	\sq(W^\lambda)_{\one} & \text{if $\lambda \in E_n$,} \\
	\fsl(W^\lambda)_{\one} & \text{if $\lambda \in F_n$.}
	\end{cases}
	\end{equation}

Set $\fs_n = \Lie(\CS_n)$ and $\fs_n' = \fD(\fs_n)$. By induction, one has $\g_{n-1} = \fs_{n-1}' + \C \cdot T_{n-1}$. Then
	\[
	W^\lambda(\g_n) \supseteq W^\lambda(\g_{n-1}) \supseteq W^\lambda(\fs_{n-1}').
	\]
In the notation of Section \ref{subsec:res-irred-super}, one has $W^\lambda = \bigoplus_{k \in \Z} W_k^\lambda$, with $W_k^\lambda \neq 0$ only if there exists a (unique) partition $\mu_k \prec \lambda$ such that $\res(\lambda/\mu_k) = k$. To simplify notation, for the rest of this section we will fix a partition $\lambda \in \Pbar(n)$, and we will write
	\[
	W^k = W_k^\lambda.
	\]
By \cref{prop:Wklambda-restriction}, if $W^k \neq 0$, then $W^k$ identifies as a $\CS_{n-1}$-supermodule with either the simple supermodule $W^{\mu_k}$, or the direct sum of $W^{\mu_k}$ and its parity shift $\Pi(W^{\mu_k})$. In any case, if $k \neq \ell$, then $W^k$ and $W^\ell$ have no simple $\CS_{n-1}$-supermodule constituents in common. This implies by the analogue of \eqref{eq:D(CSn)} for $\CS_{n-1}$ that
	\begin{equation} \label{eq:gn-1-image}
	W^\lambda(\fs_{n-1}') = \bigoplus_{k \in \Z} W^k(\fs_{n-1}'),
	\end{equation}
where $W^k(\fs_{n-1}')$ denotes the image of $\fs_{n-1}'$ in $\End(W^k)$.

Conceptually, our strategy for the proof of \cref{theorem:Wlambda(gn)} runs roughly as follows. First, we show that $W^\lambda(\g_n)$ contains a large semisimple Lie subalgebra $\fh$---specifically, a direct sum of special linear Lie algebras---over which
	\begin{align} 
	\End(W^\lambda)_{\one} &= \bigoplus_{k,\ell \in \Z} \Hom(W^k,W^\ell)_{\one} \label{eq:End(Wlambda)-odd-decomp} \\
	&= \bigoplus_{k,\ell \in \Z} \left[ \Hom(W^k_{\zero},W^\ell_{\one}) \oplus \Hom(W^k_{\one},W^\ell_{\zero}) \right] \label{eq:End(Wlambda)-odd-decomp-finer}
	\end{align}
is a semisimple $\fh$-module. Next, the transposition $s_{n-1} = (n-1,n)$ defines an element $W^\lambda(s_{n-1}) \in \End(W^\lambda)_{\one}$ that has nonzero components in various simple $\fh$-module  summands of $\End(W^\lambda)_{\one}$. Using the semisimplicity of $\fh$, we deduce that certain $\fh$-module summands of $\End(W^\lambda)_{\one}$ must be contained in the Lie superalgebra generated by $W^\lambda(\fs_{n-1}')$ and $W^\lambda(s_{n-1})$, and hence must be contained in $W^\lambda(\g_n)$. These summands in turn generate a large enough Lie superalgebra for us to deduce the inclusion \eqref{eq:sufficient-inclusion}.

\subsection{Proof of Theorem \ref{theorem:Wlambda(gn)}: the case \texorpdfstring{$\lambda \in F_n$}{lambda in Fn}} \label{subsec:proof-Fn}

\subsubsection{}

First suppose $\lambda \in F_n$. Then $W^k \neq 0$ only if $k \geq 0$, and \eqref{eq:gn-1-image} takes the form
	\begin{equation} \label{eq:gn-1-image-Fn}
	W^\lambda(\fs_{n-1}') = \bigoplus_{k \geq 0} W^k(\fs_{n-1}') = \fsl(W^0) \oplus \sq(W^1) \oplus \sq(W^2) \oplus \cdots.
	\end{equation}
The summand indexed by $k=0$ is of the form $\fsl(W^0)$ because if $\mu_0 \prec \lambda$ and $\res(\lambda/\mu_0) = 0$, then $\mu_0$ must be symmetric. Similarly, the summands indexed by integers $k > 0$ are of the form $\sq(W^k)$ because if $\mu_k \prec \lambda$ and $\res(\lambda/\mu_k) = k > 0$, then $\mu_k$ must be non-symmetric. By definition, a summand is zero if $W^k = 0$. Since $n \geq 6$, \cref{lemma:lower-bound-dim} implies that if $W^k \neq 0$, then
	\begin{equation} \label{eq:Wk-dim-geq-5}
	\dim(W^k) \geq \min\set{ 2(n-1)-2, (n-1)+1} \geq 6.
	\end{equation}

\subsubsection{} \label{subsubsec:subalgebra-f}

If $\fsl(W^0)$ is the only nonzero summand in \eqref{eq:gn-1-image-Fn}, then $W^\lambda = W^0$, and hence
	\[
	\fsl(W^\lambda) = \fsl(W^0) = W^\lambda(\fs_{n-1}') \subseteq W^\lambda(\g_n),
	\]
establishing \eqref{eq:sufficient-inclusion}. So assume that $W^k \neq 0$ for at least one value $k > 0$. For such $k$, $W^k \cong W^{\mu_k}$ is a Type Q simple $\CS_{n-1}$-supermodule, and the even and odd subspaces $W^k_{\zero}$ and $W^k_{\one}$ of $W^k$ can be identified via the odd involution $J^{\mu_k} : W^{\mu_k} \to W^{\mu_k}$. Making the identification $W^k_{\zero} \simeq W^k_{\one}$ via $J^{\mu_k}$, and writing $W_k$ for this new common space (considered just as an ordinary vector space, without any superspace structure), the diagonal maps
	\begin{equation} \label{eq:Type-Q-diagonal-maps}
	\gl(W_k) \to \gl(W^k_{\zero}) \oplus \gl(W^k_{\one}) \quad \text{and} \quad \fsl(W_k) \to \Hom(W^k_{\zero},W^k_{\one}) \oplus \Hom(W^k_{\one},W^k_{\zero})
	\end{equation}
induce vector space isomorphisms $\gl(W_k) \cong \sq(W^k)_{\zero}$ and $\fsl(W_k) \cong \sq(W^k)_{\one}$ that are compatible with the adjoint action. At the risk of confusing the reader, we will immediately change the meaning of our notation and will write $\fsl(W_k)$ to mean the evident Lie subalgebra of $\gl(W_k) \cong \sq(W^k)_{\zero}$. With this notation, we see that
	\[
	\ff := [\fsl(W^0_{\zero}) \oplus \fsl(W^0_{\one})] \oplus \fsl(W_1) \oplus \fsl(W_2) \oplus \cdots \oplus \fsl(W_{n-1})
	\]
naturally identifies with a semisimple Lie subalgebra of $W^\lambda(\fs_{n-1}')_{\zero} \subseteq W^\lambda(\g_n)$.\footnote{Since $\dim(W^k) \geq 6$ whenever $W^k \neq 0$, the nonzero summands in $\ff$ are each of the form $\fsl(m)$ for some $m \geq 3$.} Further, \eqref{eq:End(Wlambda)-odd-decomp} and \eqref{eq:End(Wlambda)-odd-decomp-finer} give $\ff$-module decompositions of $\End(W^\lambda)_{\one}$ under the adjoint action.

\subsubsection{}

We will write elements of $\End(W^k)$ in the supermatrix block form
	\begin{equation} \label{eq:block-matrix}
	\left[ \begin{array}{c|c} A & B \\ \hline C & D \end{array} \right]
	\end{equation}
where $A \in \Hom(W^k_{\zero},W^k_{\zero})$, $B \in \Hom(W^k_{\one},W^k_{\zero})$, $C \in \Hom(W^k_{\zero},W^k_{\one})$, and $D \in \Hom(W^k_{\one},W^k_{\one})$. In this notation, the inclusion $\sq(W^k) \subseteq W^\lambda(\g_n)$ for $k \geq 1$ translates into the statement that $W^\lambda(\g_n)$ contains all supermatrices such that $A = D$, $B = C$, and $\tr(B) = 0$, while the summand $\fsl(W_k)$ of the algebra $\ff$ identifies with those supermatrices such that $A = D$, $B = C = 0$, and $\tr(A) = 0$. For each $k \geq 1$, one sees that, as an $\ff$-module, $\End(W^k)_{\one}$ is the direct sum of:
	\begin{enumerate}[label=(\thesubsection.\arabic*)]
	\setcounter{enumi}{\value{equation}}
	\item a two-dimensional trivial $\ff$-submodule, spanned by the odd supermatrices\footnote{A supermatrix of the form \eqref{eq:block-matrix} is even if $B = C = 0$, and is odd if $A = D = 0$.} of the form \eqref{eq:block-matrix} such that $B$ and $C$ are arbitrary scalar matrices; and
	\item \label{item:sl(Wk)-pm} two nontrivial isomorphic simple $\ff$-modules $\fsl(W_k)^+$ and $\fsl(W_k)^-$, spanned by the odd supermatrices of the form \eqref{eq:block-matrix} such that $\tr(B) = 0$ and $C = 0$, and such that $B = 0$ and $\tr(C) = 0$, respectively. Via the projection $\ff \twoheadrightarrow \fsl(W_k)$, these simples each identify with the adjoint representation of $\fsl(W_k)$. Moreover, these two nontrivial simples do not occur in any other summands in \eqref{eq:End(Wlambda)-odd-decomp}.
	\setcounter{equation}{\value{enumi}}
	\end{enumerate}

\subsubsection{} \label{subsubsec:Fn-sl-summands}

Our first main goal is to show that $W^\lambda(\g_n)$ contains the semisimple Lie algebra
	\begin{equation} \label{eq:Lie-algebra-h}
	\fh := [ \fsl(W^0_{\zero}) \oplus \fsl(W^0_{\one}) ] \oplus [ \fsl(W^1_{\zero}) \oplus \fsl(W^1_{\one}) ] \oplus \cdots \oplus [ \fsl(W^{n-1}_{\zero}) \oplus \fsl(W^{n-1}_{\one}) ].
	\end{equation}
Given $k \geq 1$ such that $W^k \neq 0$, we will show that $[\fsl(W^k_{\zero}) \oplus \fsl(W^k_{\one})] \subseteq W^\lambda(\g_n)$ by considering the component of the map $W^\lambda(s_{n-1})$ that lies in the summand $\End(W^k)_{\one}$ of \eqref{eq:End(Wlambda)-odd-decomp}.

By definition, $W^k$ is spanned by the vectors $v_{\alpha}^+$ and $v_{\alpha}^-$ for $\alpha \in \Wbar(\lambda)$ of the form $\alpha = (\cdots, j,k)$. If $\abs{k-j} = 1$, then $W^\lambda(s_{n-1})$ acts on the vectors $v_\alpha^+$ and $v_\alpha^-$ via the matrix given in \cref{prop:si-on-weights-super}\eqref{item:si-super-trivial-action}. If $\abs{k-j} \geq 2$ and $j \neq -k$, then $W^\lambda(s_{n-1})$ maps the vectors $v_\alpha^+$ and $v_\alpha^-$ into the subspace $W^{\abs{j}} = W^\lambda_{\abs{j}}$, and $W^{\abs{j}} \neq W^k$. Thus, if $\abs{k-j} \geq 2$ and $j \neq -k$, then the action of $W^\lambda(s_{n-1})$ on $v_\alpha^+$ and $v_\alpha^-$ does not arise from a map in $\End(W^k)_{\one}$.

Next consider a weight in $\Wbar(\lambda)$ of the form $\alpha = (\alpha^\star,-k,k) = (\alpha_1,\ldots,\alpha_{n-2},-k,k)$. Weights of this form do exist: If the Young diagram of $\lambda$ has a removable box $B_k$ of residue $k \geq 1$, then by symmetry of $\lambda$ it also has a removable box $B_{-k}$ of residue $-k$, and the two boxes can be removed in either order to produce a symmetric partition $\lambda^\star$ of $n-2$. Let $T^\star$ be any standard $\lambda^\star$-tableau, and let $\alpha^\star = \alpha(T^\star) \in \calW(\lambda^\star)$ be the weight of $T^\star$. We can extend $T^\star$ to a standard $\lambda$-tableau in two ways: by putting $n-1$ in box $B_{-k}$ and $n$ in box $B_k$, to get a standard $\lambda$-tableau $T_k$ of weight $(\alpha^\star,-k,k)$, or by putting $n-1$ in box $B_k$ and $n$ in box $B_{-k}$ to get a standard $\lambda$-tableau $T_{-k}$ of weight $(\alpha^\star,k,-k)$. Every weight in $\Wbar(\lambda)$ of the form $(\alpha^\star,-k,k)$ arises in this way.

Now if $\alpha = (\alpha^\star,-k,k)$ is a weight in $\Wbar(\lambda)$, then $\gamma := (-\alpha^\star,-k,k)$ is also a weight in $\Wbar(\lambda)$ (because $\Wbar(\lambda^*)$ includes the negatives of each of its weights), and $\alpha,-\alpha,\gamma,-\gamma$ are four distinct weights in $\calW(\lambda)$. Let $\beta = -\gamma$, let $c = (k-(-k))^{-1} = 1/(2k)$, and let $w_\beta = (s_{n-1}-c) \cdot v_\alpha$ and $w_{-\beta} = (s_{n-1}+c) \cdot v_{-\alpha}$ be defined as in \cref{prop:si-on-weights-super}\eqref{item:si-super-nontrivial-action}. Then $w_\beta = a \cdot v_{-\gamma}$ for some $0 \neq a \in \C$, hence $w_{-\beta} = -\phi^\lambda(w_\beta) = -a \cdot v_{\gamma}$,
	\begin{align*}
	w_\beta^+ &= \tfrac{1}{2}(w_\beta - w_{-\beta}) = a \cdot \tfrac{1}{2}(v_{-\gamma} + v_\gamma) = a \cdot v_\gamma^+, \quad \text{and} \\
	w_\beta^- &= \tfrac{1}{2}(w_\beta + w_{-\beta}) = a \cdot \tfrac{1}{2}(v_{-\gamma} - v_\gamma) = -a \cdot v_\gamma^-.
	\end{align*}
By \cref{prop:si-on-weights-super}\eqref{item:si-super-nontrivial-action}, $W^\lambda(s_{n-1})$ leaves invariant the span of $v_\alpha^+, v_\gamma^+, v_\alpha^-, v_\gamma^-$, and acts in this homo\-geneous basis via the supermatrix
		\[
		\left[
			\begin{array}{cc|cc}
			0 & 0 & c & -(1-c^2)/a \\
			0 & 0 & a & c \\
			\hline
			c & (1-c^2)/a & 0 & 0 \\
			-a & c & 0 & 0
			\end{array}
			\right].
		\]
Combined with the observations two paragraphs ago, this implies that the component of $W^\lambda(s_{n-1})$ in $\End(W^k)_{\one}$ can be written as an odd supermatrix of the form \eqref{eq:block-matrix} such that $B \neq C$, but each pair of corresponding diagonal entries in $B$ and $C$ are equal.

Since $k \geq 1$, the partition $\mu_k \prec \lambda$ is not symmetric. Then by \eqref{eq:Wlambda(Tn)}, $W^k(T_{n-1}) = \res(\mu_k) \cdot J^{\mu_k}$, and $\res(\mu_k) = \res(\lambda) - k = -k \neq 0$. Now it follows that for some scalar $r$, the component in $\End(W^k)_{\one}$ of the operator $W^k(s_{n-1}-r \cdot T_{n-1}) = W^k(s_{n-1}) - r \cdot W^k(T_{n-1})$ has the form
	\[
	\left[ \begin{array}{c|c} 0 & B' \\ \hline C' & 0 \end{array} \right]
	\]
where $B'$ and $C'$ are nonzero traceless matrices. Recall that $\sq(W^k)_{\one} \subseteq W^\lambda(\g_n)$ consists of all supermatrices $\sm{0 & X \\ X & 0}$ such that $\tr(X) = 0$. Then $\phi := W^k(s_{n-1}-r \cdot T_{n-1}) - \sm{0 & C' \\ C' & 0}$ is an odd element of $W^\lambda(\g_n)$ whose component in $\End(W^k)_{\one}$ is a nonzero element of the $\ff$-module $\fsl(W_k)^+$ described in \ref{item:sl(Wk)-pm}. Since $\fsl(W_k)^+$ does not occur in any other summand of \eqref{eq:End(Wlambda)-odd-decomp}, the semisimplicity of the Lie algebra $\ff$ implies that the entire module $\fsl(W_k)^+$ must be contained in the $\ff$-submodule of $W^\lambda(\g_n)$ generated by $\phi$, i.e., $\fsl(W_k)^+ \subseteq W^\lambda(\g_n)$. Similarly, one gets $\fsl(W_k)^- \subseteq W^\lambda(\g_n)$. Now using the observation just after \eqref{eq:Wk-dim-geq-5} that $\dim(W^k) \geq 6$, and hence $\dim(W^k_{\zero}) = \dim(W^k_{\one}) \geq 3$, one can show that the Lie superalgebra generated by the subspaces $\fsl(W_k)^+$ and $\fsl(W_k)^-$ of $\End(W^k)$ must contain nonzero elements $x \in \fsl(W^k_{\zero})$ and $y \in \fsl(W^k_{\one})$. Then the Lie algebra generated by $x$, $y$, and $\gl(W_k) = \sq(W^k)_{\zero}$ must contain $\fsl(W^k_{\zero}) \oplus \fsl(W^k_{\one})$. Thus, $[\fsl(W^k_{\zero}) \oplus \fsl(W^k_{\one})] \subseteq W^\lambda(\g_n)$.

\subsubsection{}

We have shown that the semisimple Lie algebra $\fh$ of \eqref{eq:Lie-algebra-h} is contained in $W^\lambda(\g_n)$. Next observe that \eqref{eq:End(Wlambda)-odd-decomp-finer} is a multiplicity-free decomposition of $\End(W^\lambda)_{\one}$ into simple $\fh$-modules. In particular, each simple summand is equal to its own isotypical component. Together with the semisimplicity of $\fh$, this implies that if $\psi \in \End(W^\lambda)_{\one}$, and if $\psi$ has a nonzero component in some simple $\fh$-module summand of $\End(W^\lambda)_{\one}$, then the entire summand in question must be contained in the $\fh$-submodule of $\End(W^\lambda)_{\one}$ generated by $\psi$. By induction, we already know that $\End(W^0)_{\one} \subseteq W^\lambda(\g_n)$. Then to show that $\End(W^\lambda)_{\one} \subseteq W^\lambda(\g_n)$, it suffices to show for all $k,\ell \in \Z$ for which $W^k \neq 0$ and $W^\ell \neq 0$, and for which at least one of $k$ or $\ell$ is nonzero, that $W^\lambda(s_{n-1})$ has nonzero components in both $\Hom(W^k_{\zero},W^\ell_{\one})$ and $\Hom(W^k_{\one},W^\ell_{\zero})$. We have already established this is true when $k = \ell \geq 1$, so we may assume that $k \neq \ell$.

If $W^k \neq 0$ and $W^\ell \neq 0$, then the Young diagram of $\lambda$ has removable boxes $B_k$ and $B_\ell$ of residues $k$ and $\ell$, respectively, and these boxes can be removed in either order. Moreover, since $B_k$ and $B_\ell$ are both removable, it must be the case that $\abs{k-\ell} \geq 2$. Now reasoning as we did earlier, one can deduce that $\Wbar(\lambda)$ contains a pair of weights of the forms $\alpha = (\alpha^\star,\ell,k)$ and $\beta = (\alpha^\star,k,\ell)$. Finally, applying \cref{prop:si-on-weights-super}\eqref{item:si-super-nontrivial-action}, one sees that $W^\lambda(s_{n-1})$ has nonzero components in each of $\Hom(W^k_{\zero},W^\ell_{\one})$, $\Hom(W^k_{\one},W^\ell_{\zero})$, $\Hom(W^\ell_{\zero},W^k_{\one})$, and $\Hom(W^\ell_{\one},W^k_{\zero})$.

\subsection{Proof of Theorem \ref{theorem:Wlambda(gn)}: the case \texorpdfstring{$\lambda \in E_n$}{lambda in En}} \label{subsec:proof-En}

\subsubsection{}

Now suppose $\lambda \in E_n$. In this case one has $W^k \neq 0$ only if $\abs{k} < n$. If there is only one summand in the decomposition $W^\lambda = \bigoplus_{k \in \Z} W^k$, i.e., if $\lambda$ has only one removable box, say $B_\ell$ of residue $\ell$, then the Young diagram of $\lambda$ must be a (non-symmetric) rectangle with $B_\ell$ in its outer corner, and hence the partition $\mu_\ell$ obtained by removing $B_\ell$ must also be non-symmetric. Then by induction, $W^\lambda \cong W^{\mu_\ell}$ as a $\CS_{n-1}$-supermodule, and
	\[
	\sq(W^\lambda) = \sq(W^{\mu_\ell}) = W^{\mu_\ell}(\fs_{n-1}') = W^\lambda(\fs_{n-1}') \subseteq W^\lambda(\g_n),
	\]
establishing \eqref{eq:sufficient-inclusion}. So assume that $W^k \neq 0$ for more than one value of $k$.

\subsubsection{}

Since $\lambda$ is not symmetric, there is at most one value of $k$ such that the partition $\mu_k \prec \lambda$ is symmetric; call this value $s$ (if it exists). If $s$ exists, then $s \neq 0$, because $\lambda$ is not symmetric.

By the induction hypothesis,
	\begin{equation} \label{eq:gn-1-image-En}
	W^\lambda(\fs_{n-1}') = W^s(\fs_{n-1}') \oplus \bigoplus_{k \neq s} \sq(W^k),
	\end{equation}
where, by definition, if a symmetric partition $\mu_s \prec \lambda$ does not exist then $W^s = 0$ and the first summand is omitted. The supermodule $W^\lambda$ is equipped with the odd involution $J = J^\lambda : W^\lambda \to W^\lambda$, which restricts for each $k \in \Z$ to an odd involution $J^k : W^k \to W^k$. As in Section \ref{subsec:proof-Fn}, we make the identification $W^k_{\zero} \simeq W^k_{\one}$ via $J^k$ and write $W_k$ for the common identified space. Then for $k \neq s$ one has $\gl(W_k) \cong \sq(W^k)_{\zero}$ as in \eqref{eq:Type-Q-diagonal-maps}. For $k=s$, we see from \eqref{eq:Wmu-identification} and \eqref{eq:Pi(Wmu)-identification} that $J^s$ defines an odd isomorphism $W^{\mu_s} \simeq \Pi(W^{\mu_s})$. Then conjugation by $J^s$, $\phi \mapsto J^s \circ \phi \circ J^s$, defines an even isomorphism $\End(W^{\mu_s}) \cong \End(\Pi(W^{\mu_s}))$, and $W^s(\fs_{n-1}')$ is the image in $\End(W^s)$ of the diagonal map
	\begin{equation} \label{Ws(gn-1')-image}
	\fsl(W^{\mu_s}) \to \End(W^{\mu_s}) \oplus \End(\Pi(W^{\mu_s})).
	\end{equation}
Make the identifications $W^{\mu_s} \simeq \Pi(W^{\mu_s})$, $W^{\mu_s}_{\zero} \simeq \Pi(W^{\mu_s}_{\zero})$, and $W^{\mu_s}_{\one} \simeq \Pi(W^{\mu_s}_{\one})$ via $J^s$, and write $\Wmus$, $\Wmuszero$, and $\Wmusone$ for the common identified spaces, respectively (considered just as ordinary vector spaces, without any superspace structures).


\begin{remark} \label{remark:sl(1)=0}
It could happen that
	\begin{enumerate}
	\item \label{item:slbar-zero} $\Wmuszero \cong \Wmusone \neq 0$, but $\fsl(\Wmuszero) \cong \fsl(\Wmusone) = 0$; or that
	\item \label{item:slWk-zero} $W_k \neq 0$ for some $k \neq s$, but $\fsl(W_k) = 0$.
	\end{enumerate}
These situations occur if and only if $\dim(W^{\mu_s}) = 2$ or $\dim(W^{\mu_k}) = 2$, respectively. \cref{lemma:lower-bound-dim} implies for $n \geq 6$ that situation \eqref{item:slbar-zero} cannot occur, and implies that up to the equivalence $\lambda \sim \lambda'$, situation \eqref{item:slWk-zero} occurs only if $\lambda = (n-1,1)$. In this case, $W^{(n-1,1)} = W^{(n-1,1)}_{n-2} \oplus W^{(n-1,1)}_{-1}$, and one has $\CS_{n-1}$-super\-module isomorphisms
	\begin{align*}
	W^{(n-1,1)}_{n-2} &\cong W^{(n-2,1)}, & \dim(W^{(n-2,1)}) &= 2(n-2), \\
	W^{(n-1,1)}_{-1} &\cong W^{(n-1)}, & \dim(W^{(n-1)}) &= 2.
	\end{align*}
In any event, for $k,\ell \neq s$ the space $\Hom(W_k,W_\ell)$ remains a simple $\fsl(W_k) \oplus \fsl(W_\ell)$-module even if one of $W_k$ or $W_\ell$ is one-dimensional (hence even if one of $\fsl(W_k)$ or $\fsl(W_\ell)$ is zero).
\end{remark}

\subsubsection{}

Now from \eqref{eq:gn-1-image-En}, we see that the semisimple Lie algebra
	\[
	\fh := \left[ \fsl(\Wmuszero) \oplus \fsl(\Wmusone) \right] \oplus \bigoplus_{k \neq s} \fsl(W_k)
	\]
identifies with a subalgebra of $W^\lambda(\fs_{n-1}')_{\zero} \subseteq W^\lambda(\g_n)$. Further, \eqref{eq:End(Wlambda)-odd-decomp} and \eqref{eq:End(Wlambda)-odd-decomp-finer} give $\fh$-module decompositions of $\End(W^\lambda)_{\one}$ under the adjoint action. The set $W^\lambda(\g_n)$ is contained in
	\[
	\fq(W^\lambda) = \End(W^\lambda)^J := \smallset{ \theta \in \End(W^\lambda) : J \circ \theta \circ J = \theta},
	\]
and the decomposition \eqref{eq:End(Wlambda)-odd-decomp} gives rise to the corresponding decomposition of $J$-invariants
	\begin{equation}
	\End(W^\lambda)^J_{\one} = \bigoplus_{k,\ell \in \Z} \Hom(W^k,W^\ell)^J_{\one}.
	\end{equation}
For $s \notin \set{k,\ell}$, one sees that the diagonal map
	\[
	\Hom(W_k,W_\ell) \to \Hom(W^k_{\zero},W^\ell_{\one}) \oplus \Hom(W^k_{\one},W^\ell_{\zero})
	\]
induces an $\fh$-module isomorphism $\Hom(W_k,W_\ell) \cong \Hom(W^k,W^\ell)^J_{\one}$. Similarly, for $k \neq s$ one sees that the diagonal maps
	\begin{align*}
	\Hom(\Wmuszero,W_k) &\to \Hom(W^{\mu_s}_{\zero},W^k_{\one}) \oplus \Hom(\Pi(W^{\mu_s}_{\zero}),W^k_{\zero}), \quad \text{and} \\
	\Hom(\Wmusone,W_k) &\to \Hom(W^{\mu_s}_{\one},W^k_{\zero}) \oplus \Hom(\Pi(W^{\mu_s}_{\one}),W^k_{\one})
	\end{align*}
induce an $\fh$-module isomorphism
	\[
	\Hom(\Wmuszero,W_k) \oplus \Hom(\Wmusone,W_k) \cong \Hom(W^s,W^k)^J_{\one}.
	\]
An analogous description holds for $\Hom(W^k,W^s)^J_{\one}$. Finally, as an $\fh$-module,
	\begin{equation*}
	\begin{split}
	\Hom(W^s,W^s)^J_{\one} \cong& \Hom(\Wmuszero,\Wmusone) \oplus \Hom(\Wmusone,\Wmuszero) \\
	&\oplus \Hom(\Wmuszero,\Wmuszero) \oplus \Hom(\Wmusone,\Wmusone),
	\end{split}
	\end{equation*}
where the summands on the right side of the isomorphism are identified with the images of the corresponding diagonal maps
	\begin{align}
	\Hom(\Wmuszero,\Wmusone) &\to \Hom(W^{\mu_s}_{\zero},W^{\mu_s}_{\one}) \oplus \Hom(\Pi(W^{\mu_s}_{\zero}),\Pi(W^{\mu_s}_{\one})), \\
	\Hom(\Wmusone,\Wmuszero) &\to \Hom(W^{\mu_s}_{\one},W^{\mu_s}_{\zero}) \oplus \Hom(\Pi(W^{\mu_s}_{\one}),\Pi(W^{\mu_s}_{\zero})), \\
	\Hom(\Wmuszero,\Wmuszero) &\to  \Hom(W^{\mu_s}_{\zero},\Pi(W^{\mu_s}_{\zero})) \oplus \Hom(\Pi(W^{\mu_s}_{\zero}),W^{\mu_s}_{\zero}), \label{End(Wmuszero)-diagonal} \\
	\Hom(\Wmusone,\Wmusone) &\to  \Hom(W^{\mu_s}_{\one},\Pi(W^{\mu_s}_{\one})) \oplus \Hom(\Pi(W^{\mu_s}_{\one}),W^{\mu_s}_{\one}). \label{End(Wmusone)-diagonal}
	\end{align}

\subsubsection{Simple constituents} \label{subsubsec:End(Wlambda)Jone-irreds}

Altogether, $\End(W^\lambda)^J_{\one}$ admits the $\fh$-module decomposition
	\begin{equation} \label{eq:End(Wlambda)Jone}
	\begin{split}
	&\End(W^\lambda)^J_{\one} \cong \Big[ \bigoplus_{s \notin \set{k,\ell}} \Hom(W_k,W_\ell) \Big] \\
	&\oplus \Big[ \bigoplus_{k \neq s} \Hom(\Wmuszero,W_k) \oplus \Hom(\Wmusone,W_k) \oplus \Hom(W_k,\Wmuszero) \oplus \Hom(W_k,\Wmusone) \Big] \\
	&\oplus \Big[ \Hom(\Wmuszero,\Wmusone) \oplus \Hom(\Wmusone,\Wmuszero) \oplus \End(\Wmuszero) \oplus \End(\Wmusone) \Big].
	\end{split}
	\end{equation}
For $k \neq s$, the term $\End(W_k)$ in \eqref{eq:End(Wlambda)Jone} is either simply a one-dimensional trivial $\fh$-module, if $\dim(W_k)=1$, or else is the direct sum of a one-dimensional trivial $\fh$-module and a copy of the adjoint module for $\fsl(W_k)$, the latter of which is contained in $W^\lambda(\fs_{n-1}')_{\one} \subseteq W^\lambda(\g_n)$ by \eqref{eq:gn-1-image-En}. By \eqref{Ws(gn-1')-image}, the summands $\Hom(\Wmuszero,\Wmusone)$ and $\Hom(\Wmusone,\Wmuszero)$ are also contained in $W^\lambda(\g_n)$. The summands $\End(\Wmuszero)$ and $\End(\Wmusone)$ in \eqref{eq:End(Wlambda)Jone} are each direct sums of a one-dimensional trivial $\fh$-module and a copy of the adjoint representation for $\fsl(\Wmuszero)$ and $\fsl(\Wmusone)$, respectively. The remaining nonzero summands in \eqref{eq:End(Wlambda)Jone} are each nontrivial simple $\fh$-modules. Overall, the non-trivial simple $\fh$-modules that occur in $\End(W^\lambda)^J_{\one}$ each do so with multiplicity one.  

\subsubsection{}

If $k$, $\ell$, and $s$ are distinct, and if $W^k$ and $W^\ell$ are both nonzero, then one can argue as in the last two paragraphs of Section \ref{subsec:proof-Fn} to show first that $W^\lambda(s_{n-1})$ has nonzero components in the simple $\fh$-module summands $\Hom(W^k,W^\ell)^J_{\one}$ and $\Hom(W^\ell,W^k)^J_{\one}$ of $\End(W^\lambda)^J_{\one}$, and then to deduce that these summands must both be contained in $W^\lambda(\g_n)$.

\subsubsection{The case \texorpdfstring{$W^s = 0$}{Ws = 0}} \label{subsubsec:Ws=0}

If $W^s = 0$, then the previous paragraph together with our observations in Section \ref{subsubsec:End(Wlambda)Jone-irreds} imply that each non-trivial $\fh$-module constituent of $\End(W^\lambda)^J_{\one}$ is contained in $W^\lambda(\g_n)$. Identifying $\sq(W^\lambda)$ with supermatrices as in \eqref{eq:sq(n)-supermatrices}, it implies that
	\begin{equation} \label{eq:odd-no-diagonal}
	\set{ \left[ \begin{array}{c|c} 0 & B \\ \hline B & 0 \end{array} \right] : \text{the diagonal entries of $B$ are all zero} }
	\end{equation}
is contained in $W^\lambda(\g_n)$. More precisely, the direct sum decomposition \eqref{eq:End(Wlambda)Jone} induces a block decomposition of the matrix $B$ such that the diagonal blocks correspond to $\oplus_{k \in \Z} \End(W_k)$, and we deduce that $W^\lambda(\g_n)$ contains the (larger) set of all matrices of the form $\sm{0 & B \\ B & 0}$ such that these diagonal blocks each individually have trace zero. By \cref{lemma:lower-bound-dim} and the assumption that $W^k \neq 0$ for more than one value of $k$, we have $\dim(W^\lambda) \geq 2n-2 \geq 10$. Then the inclusion $\sq(W^\lambda) \subseteq W^\lambda(\g_n)$ in the case where $W^s = 0$ is obtained from the following lemma:

\begin{lemma} \label{lemma:sq(m)-generated}
If $m \geq 5$, then $\sq(m)$ is generated as a Lie superalgebra by the identity matrix $I_{m|m}$ and the set \eqref{eq:odd-no-diagonal}.
\end{lemma}

\begin{proof}
Let $\g \subseteq \sq(m)$ be the Lie superalgebra generated by the identity matrix $I_{m|m}$ and the set \eqref{eq:odd-no-diagonal}. First show that all simple root vectors in $\fsl(m) \subseteq \gl(m) \cong \sq(m)_{\zero}$ are elements of $\g$ by taking Lie brackets between root vectors in the set \eqref{eq:odd-no-diagonal}. The simple root vectors generate $\fsl(m)$ as a Lie algebra, and together with the identity matrix they generate all of $\gl(m) \cong \sq(m)_{\zero}$. Then $\sq(m)_{\zero} \subseteq \g$. Now since $\sq(m)_{\one}$ is irreducible under the adjoint action of $\sq(m)_{\zero}$, one deduces that $\sq(m)_{\one} \subseteq \g$ as well, and hence $\g = \sq(m)$.
\end{proof}

\subsubsection{The case \texorpdfstring{$W^s \neq 0$}{Ws neq 0}}

Now suppose that $k \neq s$ and that $W^k$ and $W^s$ are both nonzero. Let $B_k$ and $B_s$ be the removable boxes of residues $k$ and $s$ in the Young diagram of $\lambda$. Recall that $s \neq 0$ because $\lambda$ is not symmetric. Since $B_k$ and $B_s$ are both removable, we can remove $B_s$ and then $B_k$, showing that $\mu_s$ has a removable box of residue $k$. Then by symmetry, $\mu_s$ must have a removable box $B_{-k}$ of residue $-k$. (It may happen that $k = 0$, in which case $B_{-k} = B_k$.) This implies that $-k \neq s$, because a new box of residue $s$ would be removable from $\mu_s$ only if there had originally been boxes both immediately above and to the left of $B_s$ in the Young diagram of $\lambda$, and if both of those boxes had already been removed. In other words, it requires at least two intermediate steps to remove two boxes of the same residue from $\lambda$. Now since the boxes $B_s$ and $B_k$ are both removable, it follows that $\abs{s-k} \geq 2$. And since $-k \neq s$, then $\abs{s-(-k)} \geq 1$. If $\abs{s-(-k)} = 1$, then $-k = s \pm 1$, and $B_{-k}$ is located either immediately above (if $-k = s+1$) or immediately to the left (if $-k=s-1$) of $B_s$. Either way, the box $B_{-k}$ is not removable from the Young diagram of $\lambda$. If $\abs{s-(-k)} \geq 2$, then the box $B_{-k}$ must be removable from the Young diagram of $\lambda$, and hence $\lambda$ has removable boxes of residues $s$, $k$, and $-k$ (the latter two being the same box, if $k = 0$), and the spaces $W^s$, $W^k$, and $W^{-k}$ are each nonzero.

We want to show that $W^\lambda(s_{n-1})$ has nonzero components in each of the terms in the second line of \eqref{eq:End(Wlambda)Jone}. We consider separately the cases $\abs{s - (-k)} = 1$ and $\abs{s - (-k)} \geq 2$.

\subsubsection{The case \texorpdfstring{$\abs{s-(-k)}=1$}{|s-(-k)|=1}} \label{subsubsec:|s-(-k)|=1}

First suppose that $\abs{s - (-k)} = 1$. Reasoning along lines similar to those in Section \ref{subsubsec:Fn-sl-summands}, one can find (distinct) weights in $\calW(\lambda)$ of the forms
	\begin{equation} \label{eq:three-weights}
	\begin{aligned}
	\alpha &= (\alpha^\star,k,s), &\qquad \alpha' &= (-\alpha^\star,-k,s), &\qquad \beta &= (\alpha^\star,s,k).
	\end{aligned}
	\end{equation}
Then $\gamma := (\alpha^\star,k)$ and $-\gamma = (-\alpha^\star,-k)$ are elements of $\calW(\mu_s)$. After possibly replacing $\lambda$ with $\lambda'$, we may assume that $k \geq 0$ and that $\gamma \in \Wbar(\mu_s)$; i.e., we may assume that $\gamma$ is the `positive' element of the pair $\pm \gamma$.	By our conventions in Section \ref{subsec:weight-simple-super}, each of the three weights in \eqref{eq:three-weights} is the positive element of a pair in $\Wbar(\lambda)$.

The module $S^{\mu_s}$ occurs canonically as a $\CS_{n-1}$-module summand in $S^\lambda$ (as the sum of the weight spaces whose weights end in the integer $s$) and as a $\CS_{n-1}$-module summand in $S^{\lambda'}$ (as the sum of the weight spaces whose weights end in the integer $-s$). Identifying $S^{\mu_s}$ with a summand in $S^\lambda$, the weight vectors $v_\alpha,v_{\alpha'} \in S^\lambda$ restrict to a pair of weight vectors $u_\gamma,u_{-\gamma} \in S^{\mu_s}$ of weights $\gamma$ and $-\gamma$, respectively. (We use the letter $u$ rather than $v$ to indicate when we are considering a vector's restriction to $S^{\mu_s}$.) Rescaling $v_{\alpha'}$ if necessary, we may assume that $u_{-\gamma} = \phi^{\mu_s}(u_\gamma)$ as in \eqref{eq:compatible-negative-weight-vector}. Next, since the map $\phi^\lambda: S^\lambda \to S^{\lambda'}$ restricts for each $\nu \in \calW(\lambda)$ to a linear isomorphism $S^{\lambda}_\nu \to S^{\lambda'}_{-\nu}$, we see that $\phi^\lambda$ maps the copy of $S^{\mu_s}$ in $S^\lambda$ onto the copy of $S^{\mu_s}$ in $S^{\lambda'}$. Then rescaling $\phi^\lambda$ if necessary, we may assume that $\phi^\lambda$ restricts to the map $\phi^{\mu_s}: S^{\mu_s} \to S^{\mu_s}$ specified via \cref{conv:associator-scaling}.

The summand $S^{\mu_s} \subset S^\lambda$ admits the (\emph{non}-super) $\CS_{n-1}$-decomposition $S^{\mu_s} = S^{\mu_s^+} \oplus S^{\mu_s^-}$. Under this decomposition, one gets $u_\gamma = u_\gamma^+ + u_\gamma^-$ and $u_{-\gamma} = u_\gamma^+ - u_\gamma^-$ in $S^{\mu_s} \subset S^\lambda$, where
	\begin{align} \label{eq:u-gamma-pm}
	u_\gamma^+ &= \tfrac{1}{2}\left( u_\gamma + u_{-\gamma} \right) = \tfrac{1}{2}\left( v_\alpha + v_{\alpha'} \right) \quad \text{and} \quad u_\gamma^- = \tfrac{1}{2} \left( u_\gamma - u_{-\gamma} \right) = \tfrac{1}{2}\left( v_\alpha - v_{\alpha'} \right);
	\end{align}
cf.\ \eqref{eq:Fn-valpha-pm}. Then under the $\CS_{n-1}$-supermodule identification $W^s \cong W^{\mu_s} \oplus \Pi(W^{\mu_s})$ given by \eqref{eq:Wmu-identification} and \eqref{eq:Pi(Wmu)-identification}, one has
	\begin{equation} \label{eq:W^s-summand-vectors}
	\begin{split}
	W^{\mu_s}_{\zero}		&\ni u_\gamma^+ + \phi^\lambda(u_\gamma^+) 
	= \tfrac{1}{2}\left( v_\alpha + v_{\alpha'} \right) 
	+ \tfrac{1}{2}\left( v_{-\alpha} + v_{-\alpha'} \right) 
	= v_\alpha^+ + v_{\alpha'}^+, \\
	W^{\mu_s}_{\one}		&\ni u_\gamma^- - \phi^\lambda(u_\gamma^-) 
	= \tfrac{1}{2}\left( v_\alpha - v_{\alpha'} \right) 
	- \tfrac{1}{2}\left( v_{-\alpha} - v_{-\alpha'} \right) 
	= v_\alpha^- - v_{\alpha'}^-, \\
	\Pi(W^{\mu_s}_{\zero})	&\ni u_\gamma^+ - \phi^\lambda(u_\gamma^+)
	= \tfrac{1}{2}\left( v_\alpha + v_{\alpha'} \right) 
	- \tfrac{1}{2}\left( v_{-\alpha} + v_{-\alpha'} \right) 
	= v_\alpha^- + v_{\alpha'}^-, \\
	\Pi(W^{\mu_s}_{\one})	&\ni u_\gamma^- + \phi^\lambda(u_\gamma^-)
	= \tfrac{1}{2}\left( v_\alpha - v_{\alpha'} \right) 
	+ \tfrac{1}{2}\left( v_{-\alpha} - v_{-\alpha'} \right) 
	= v_\alpha^+ - v_{\alpha'}^+.
	\end{split}
	\end{equation}

Set $c = (s-k)^{-1}$ (recall that $\abs{s-k} \geq 2$), let $w_\beta = (s_{n-1}-c) \cdot v_\alpha$, and let the auxiliary vectors $w_\beta^+,w_\beta^- \in W^k$ be defined as in \cref{prop:si-on-weights-super}\eqref{item:si-super-nontrivial-action}. Then $W^\lambda(s_{n-1})$ leaves invariant the span of the homogeneous vectors
	\begin{gather}
	\set{w_\beta^+,\; u_\gamma^+ + \phi^\lambda(u_\gamma^+),\; u_\gamma^- + \phi^\lambda(u_\gamma^-),\; w_\beta^-,\; u_\gamma^+ - \phi^\lambda(u_\gamma^+),\; u_\gamma^- - \phi^\lambda(u_\gamma^-)} \notag \\
	= \set{w_\beta^+,\; v_\alpha^+ + v_{\alpha'}^+,\; v_\alpha^+ - v_{\alpha'}^+,\; w_\beta^-,\; v_\alpha^- + v_{\alpha'}^-,\; v_{\alpha}^- - v_{\alpha'}^- }. \label{eq:|s+k|=1-basis}
	\end{gather}
Specifically, let $\epsilon \in \set{+1,-1}$ be the scalar such that $s_{n-1} \cdot v_{\alpha'} = \epsilon \cdot v_{\alpha'}$ in \cref{prop:si-on-weights}(\ref{item:alphai-pm-alphai+1}). Then applying \cref{prop:si-on-weights-super}, one can show that $W^\lambda(s_{n-1})$ acts in the homogeneous basis \eqref{eq:|s+k|=1-basis} via the matrix in Figure \ref{fig:6x6-matrix}. This shows that $W^\lambda(s_{n-1})$ has nonzero components in each of the terms in the second line of \eqref{eq:End(Wlambda)Jone}. Then by the semisimplicity of $\fh$, and by the fact that all nontrivial simple $\fh$-module summands in $\End(W^\lambda)^J_{\one}$ occur with multiplicity one, we conclude that each of the summands in the second line of \eqref{eq:End(Wlambda)Jone} must be contained in $W^\lambda(\g_n)$.

	\begin{figure}
	\begin{equation*} 
	\left[
	\renewcommand{\arraystretch}{1.1}
	\begin{array}{ccc|ccc}
	0 & 0 & 0 & -c & 1 & 1 \\
	0 & 0 & 0 & \frac{1-c^2}{2} & \frac{c+\epsilon}{2} & \frac{c-\epsilon}{2} \\
	0 & 0 & 0 & \frac{1-c^2}{2} & \frac{c-\epsilon}{2} & \frac{c+\epsilon}{2} \\
	\hline
	-c & 1 & 1 & 0 & 0 & 0 \\
	\frac{1-c^2}{2} & \frac{c+\epsilon}{2} & \frac{c-\epsilon}{2} & 0 & 0 & 0 \\
	\frac{1-c^2}{2} & \frac{c-\epsilon}{2} & \frac{c+\epsilon}{2} & 0 & 0 & 0
	\end{array}
	\right]
	\end{equation*}
	\caption{Matrix for the case $W^s \neq 0$ and $\abs{s-(-k)} = 1$.} \label{fig:6x6-matrix} 
	\end{figure}

\subsubsection{The case \texorpdfstring{$\abs{s-(-k)} \geq 2$}{|s-(-k)| >= 2}}

Now suppose $\abs{s - (-k)} \geq 2$. In this case the Young diagram of $\lambda$ has removable boxes of residues $s$, $k$, and $-k$, and we can argue as in Section \ref{subsubsec:Fn-sl-summands} to see that $\calW(\lambda)$ contains (distinct) weights of the forms
	\begin{equation} \label{eq:four-weights}
	\begin{aligned}
	\alpha &= (\alpha^\star,k,s), &\qquad \alpha' &= (-\alpha^\star,-k,s), \\
	\beta &= (\alpha^\star,s,k), &\qquad  \beta' &= (-\alpha^\star,s,-k).
	\end{aligned}
	\end{equation}
One now repeats word-for-word the reasoning in first three paragraphs of Section \ref{subsubsec:|s-(-k)|=1} to define weight vectors $v_\alpha$, $v_{\alpha'}$, $u_\gamma$, $u_{-\gamma}$ that satisfy the relations \eqref{eq:u-gamma-pm} and \eqref{eq:W^s-summand-vectors}.

Set $c = (s-k)^{-1}$ and $d = (s+k)^{-1}$. Let $w_\beta = (s_{n-1}-c) \cdot v_\alpha$, let $w_{\beta'} = (s_{n-1}-d) \cdot v_{\alpha'}$, and let the auxiliary vectors $w_\beta^+,w_\beta^- \in W^k$ and $w_{\beta'}^+,w_{\beta'}^- \in W^{-k}$ be defined as in \cref{prop:si-on-weights-super}\eqref{item:si-super-nontrivial-action}. Then $W^\lambda(s_{n-1})$ leaves invariant the span of the homogeneous vectors
	\begin{gather}
	\set{ w_\beta^+,\; w_{\beta'}^+,\; u_\gamma^+ + \phi^\lambda(u_\gamma^+),\; u_\gamma^- + \phi^\lambda(u_\gamma^-),\; w_\beta^-,\; w_{\beta'}^-,\; u_\gamma^+ - \phi^\lambda(u_\gamma^+),\; u_\gamma^- - \phi^\lambda(u_\gamma^-) } \notag \\
	= \set{w_\beta^+,\; w_{\beta'}^+,\; v_\alpha^+ + v_{\alpha'}^+,\; v_\alpha^+ - v_{\alpha'}^+,\; w_\beta^-,\; w_{\beta'}^-,\; v_\alpha^- + v_{\alpha'}^-,\; v_{\alpha}^- - v_{\alpha'}^- }. \label{eq:|s+k|>=2-basis}
	\end{gather}
Applying \cref{prop:si-on-weights-super}, one can show that $W^\lambda(s_{n-1})$ acts in the homogeneous basis \eqref{eq:|s+k|>=2-basis} via the matrix in Figure \ref{fig:big-matrix}. This shows (for both $k$ and $-k$) that $W^\lambda(s_{n-1})$ has nonzero components in each of the terms in the second line of \eqref{eq:End(Wlambda)Jone}. Then by the semisimplicity of $\fh$, and by the fact that all nontrivial simple $\fh$-module summands in $\End(W^\lambda)^J_{\one}$ occur with multiplicity one, we conclude that each of the summands in the second line of \eqref{eq:End(Wlambda)Jone} must be contained in $W^\lambda(\g_n)$.

	\begin{figure}
	\begin{equation*} 
	\left[
	\renewcommand{\arraystretch}{1.1}
	\begin{array}{cccc|cccc}
	0 & 0 & 0 & 0 & -c & 0 & 1 & 1 \\
	0 & 0 & 0 & 0 & 0 & -d & 1 & -1 \\
	0 & 0 & 0 & 0 & \frac{1-c^2}{2} & \frac{1-d^2}{2} & \frac{c+d}{2} & \frac{c-d}{2} \\
	0 & 0 & 0 & 0 & \frac{1-c^2}{2} & \frac{d^2-1}{2} & \frac{c-d}{2} & \frac{c+d}{2} \\
	\hline
	-c & 0 & 1 & 1 & 0 & 0 & 0 & 0 \\
	0 & -d & 1 & -1 & 0 & 0 & 0 & 0 \\
	\frac{1-c^2}{2} & \frac{1-d^2}{2} & \frac{c+d}{2} & \frac{c-d}{2} & 0 & 0 & 0 & 0 \\
	\frac{1-c^2}{2} & \frac{d^2-1}{2} & \frac{c-d}{2} & \frac{c+d}{2} & 0 & 0 & 0 & 0
	\end{array}
	\right]
	\end{equation*}
	\caption{Matrix for the case $W^s \neq 0$, $\abs{s-(-k)} \geq 2$.} \label{fig:big-matrix}
	\end{figure}

\subsubsection{The case \texorpdfstring{$W^s \neq 0$}{Ws neq 0}, concluded}

Now the only nontrivial $\fh$-module constituents of \eqref{eq:End(Wlambda)Jone} that we have not yet shown are contained in $W^\lambda(\g_n)$ are the copies of the adjoint representations of $\fsl(\Wmuszero)$ and $\fsl(\Wmusone)$ in $\End(\Wmuszero)$ and $\End(\Wmusone)$, respectively. Once we show that these simple constituents are contained in $W^\lambda(\g_n)$, we can then argue as in Section \ref{subsubsec:Ws=0}, using \cref{lemma:sq(m)-generated}, to conclude that $\sq(W^\lambda) \subseteq W^\lambda(\g_n)$.

We know that $W^\lambda(\g_n)$ contains the terms $\Hom(\Wmuszero,W_k)$ and $\Hom(W_k,\Wmusone)$ from \eqref{eq:End(Wlambda)Jone}. Like all of the terms in \eqref{eq:End(Wlambda)Jone}, these two terms are concentrated in odd superdegree. We also know from \eqref{Ws(gn-1')-image} that $W^s(\fs_{n-1}') \subset W^\lambda(\g_n)$ contains a copy of $\Hom(\Wmusone,\Wmuszero)$, also concentrated in odd superdegree, equal to the image of the diagonal map
	\[
	\Hom(\Wmusone,\Wmuszero) \to \Hom(W^{\mu_s}_{\one},W^{\mu_s}_{\zero}) \oplus \Hom(\Pi(W^{\mu_s}_{\one}),\Pi(W^{\mu_s}_{\zero})).
	\]
By \cref{remark:sl(1)=0} and \cref{lemma:lower-bound-dim}, we know that $W_k$, $\Wmuszero$, and $\Wmusone$ are each at least 3-dimen\-sional. Now one can choose appropriate `matrix units'
	\[
	x \in \Hom(\Wmusone,\Wmuszero), \qquad y \in \Hom(W_k,\Wmusone), \qquad z \in \Hom(\Wmuszero,W_k)
	\]
such that the Lie bracket $[x,[y,z]]$ is a nonzero element---of $W^\lambda(\g_n)$---in the subspace $\fsl(\Wmuszero) \subset \End(\Wmuszero) \subset \End(W^s)^J_{\one}$. Then by the irreducibility of the adjoint representation $\fsl(\Wmuszero)$, the $\fh$-submodule of $W^\lambda(\g_n)$ generated by $[x,[y,z]]$ must be equal to all of $\fsl(\Wmuszero)$. Similarly, one can show that $W^\lambda(\g_n)$ contains the subspace $\fsl(\Wmusone) \subset \End(\Wmusone) \subset \End(W^s)^J_{\one}$.

\subsection{First consequence of Theorem \ref{theorem:Wlambda(gn)}} \label{subsec:first-consequences}

\begin{corollary} \label{cor:g0-reductive}
Let $n \geq 2$, and set $\g = \g_n$.
	\begin{enumerate}
	\item For each $\lambda \in E_n \cup F_n$, the supermodule $W^\lambda$ is semisimple as a $\gzero$-module.
	\item $\gzero$ is a reductive Lie algebra. 
	\end{enumerate}
In particular, $\gzero = Z(\gzero) \oplus \fD(\gzero)$, where $Z(\gzero)$ is the center of $\gzero$, and $\fD(\gzero)$ is semisimple.
\end{corollary}

\begin{proof}
By \cref{theorem:Wlambda(gn)},
	\[
	W^\lambda(\gzero) = W^\lambda(\g)_{\zero} = \begin{cases}
	\sq(W^\lambda)_{\zero} & \text{if $\lambda \in E_n$,} \\
	\fsl(W^\lambda)_{\zero} & \text{if $\lambda \in F_n$}.
	\end{cases}
	\]
In either case, this implies that $W^\lambda_{\zero}$ and $W^\lambda_{\one}$ are simple $\gzero$-modules, and hence $W^\lambda$ is semisimple. Now $\bigoplus_{\lambda \in E_n \cup F_n} W^\lambda$ is a faithful, finite-dimensional, semisimple $\gzero$-module. Then $\gzero$ is reductive, $\fD(\gzero)$ is semisimple, and $\gzero = Z(\gzero) \oplus \fD(\gzero)$, by Proposition 5 of \cite[Chapter I, \S6, no.\ 4]{Bourbaki:1998}.
\end{proof}

\begin{remark}
In general, $\fD(\gzero) \neq \fD(\g)_{\zero}$.
\end{remark}

\begin{corollary} \label{cor:Wlambda(D)}
Let $n \geq 2$. For $\lambda \in E_n$, make the identification $W^\lambda_{\zero} \simeq W^\lambda_{\one}$ via the odd involution $J^\lambda : W^\lambda \to W^\lambda$, and write $W_\lambda$ for the common identified space, as in Section \ref{subsubsec:subalgebra-f}. Then
	\begin{equation} \label{eq:Wlambda(D)}
	W^\lambda(\fD(\gzero)) = \fD(W^\lambda(\g)_{\zero}) = \begin{cases}
	\fsl(W_\lambda) & \text{if $\lambda \in E_n$,} \\
	\fsl(W^\lambda_{\zero}) \oplus \fsl(W^\lambda_{\one}) & \text{if $\lambda \in F_n$,}
	\end{cases}
	\end{equation}
where $\fsl(W_\lambda)$ denotes the diagonally embedded copy of $\fsl(W_\lambda)$ in $\gl(W^\lambda_{\zero}) \oplus \gl(W^\lambda_{\one})$, as in \eqref{eq:Type-Q-diagonal-maps}. The homogeneous subspaces of $W^\lambda$ are submodules for the action of $\CA_n$, and hence also for the action of $\fD(\gzero)$; if $\lambda \in E_n$ these submodules are both isomorphic to the simple $\abs{\CA_n}$-module $S^\lambda$, while if $\lambda \in F_n$ they are isomorphic to the simple $\abs{\CA_n}$-modules $S^{\lambda^+}$ and $S^{\lambda^-}$.
\end{corollary}

\begin{proof}
Direct calculation from \cref{theorem:Wlambda(gn)}.
\end{proof}

\subsection{Detecting isomorphisms} \label{subsec:detecting-isos}

\begin{lemma} \label{lemma:submodule-equivalent}
Let $n \geq 5$, set $\g = \g_n$, and let $W$ be a finite-dimensional $\CA_n$-module. Then the following are equivalent for a subspace $V \subseteq W$:
	\begin{enumerate}
	\item \label{item:CAn-submodule} $V$ is a $\CA_n$-submodule.
	\item \label{item:g0-submodule} $V$ is a submodule for the action of the Lie subalgebra $\gzero \subseteq \CA_n$.
	\end{enumerate}
\end{lemma}
\begin{proof}
The fact that \eqref{item:CAn-submodule} and \eqref{item:g0-submodule} are equivalent is immediate from the observation in the proof of \cref{lemma:lie-algebra-centers} that for $n \geq 5$, $\gzero$ contains a set of associative algebra generators for $\CA_n$.
\end{proof}

\begin{lemma} \label{lemma:An-simple-implies-D(g0)-simple}
Let $n \geq 5$, set $\g = \g_n$, and let $W$ be a simple $\CA_n$-module. Then $W$ is simple as a module for the Lie algebra $\fD(\gzero)$.
\end{lemma}

\begin{proof}
Let $V \subseteq W$ be a nonzero $\fD(\gzero)$-submodule of $W$. Since $Z(\gzero) \subseteq Z(\CA_n)$ by \cref{lemma:lie-algebra-centers}, each element $z \in Z(\gzero)$ acts on $W$ as a scalar multiple of the identity, by Schur's Lemma. Then $V$ is closed under the action of $\gzero = Z(\gzero) \oplus \fD(\gzero)$ (where the equality holds by \cref{cor:g0-reductive}). This implies by \cref{lemma:submodule-equivalent} that $V$ is a nonzero $\CA_n$-submodule of $W$, and hence $V = W$. Thus $W$ is simple as a $\fD(\gzero)$-module.
\end{proof}

\begin{proposition} \label{prop:irred-equivalent}
Let $n \geq 5$, set $\g = \g_n$, and let $V_1$ and $V_2$ be two simple $\CA_n$-modules.
Then the following are equivalent:
	\begin{enumerate}
	\item \label{item:An-iso} $V_1$ and $V_2$ are isomorphic as $\CA_n$-modules.
	\item \label{item:gzero-iso} $V_1$ and $V_2$ are isomorphic as modules over the Lie subalgebra $\gzero \subseteq \CA_n$.
	\item \label{item:D(gzero)-iso} $V_1$ and $V_2$ are isomorphic as modules over the Lie subalgebra $\fD(\gzero) \subseteq \CA_n$.
	\end{enumerate}
\end{proposition}

\begin{proof}
Our argument is an adaptation of the proof of \cite[Proposition 2]{Marin:2007}. The fact that \eqref{item:An-iso} implies \eqref{item:gzero-iso}, and that \eqref{item:gzero-iso} implies \eqref{item:D(gzero)-iso}, is evident. We will show that \eqref{item:D(gzero)-iso} implies \eqref{item:An-iso}. For $n \geq 5$, the trivial module is the unique one-dimensional $\CA_n$-module (see \cite[Theorem 2.5.15]{James:1981}), so we may assume that $V_1$ and $V_2$ are each of dimension at least $2$.

Suppose $\phi: V_1 \to V_2$ is an isomorphism of $\fD(\gzero)$-modules, and let $\rho_1: \CA_n \to \End(V_1)$ and $\rho_2: \CA_n \to \End(V_2)$ be the structure maps for $V_1$ and $V_2$, respectively. Then for all $x \in \fD(\gzero)$, one has $\phi \circ \rho_1(x) = \rho_2(x) \circ \phi$, or equivalently, $\rho_2(x) = \phi \circ \rho_1(x) \circ \phi^{-1}$. Let $s = (i,j)(k,\ell)$ be a generator of $A_n$ from the set \eqref{eq:An-generators}, and set $T = p(s) = \frac{2}{n!} \sum_{\sigma \in A_n} \sigma s \sigma^{-1}$. The elements of the set \eqref{eq:An-generators} form a single conjugacy class in $A_n$ (because the cycle type does not consist of distinct odd integers), so $T$ is independent of the particular choice of $s$. Since $T$ is central in $\CA_n$, Schur's Lemma implies that $\rho_1(T) = c_1  \id_{V_1}$ and $\rho_2(T) = c_2  \id_{V_2}$ for some scalars $c_1,c_2 \in \C$, which also do not depend on the choice of $s$. We have $p(s-T) = p(s)-p(T) = T - T = 0$, so $s-T \in \fD(\gzero)$ by \cref{cor:g0-reductive} and \cref{lemma:lie-algebra-centers}. Then
	\begin{align*}
	\rho_2(s) - c_2  \id_{V_2} &= \rho_2(s-T) \\
	&= \phi \circ \rho_1(s-T) \circ \phi^{-1} \\
	&= \phi \circ \rho_1(s) \circ \phi^{-1} - \phi \circ (c_1  \id_{V_1}) \circ \phi^{-1} \\
	&= \phi \circ \rho_1(s) \circ \phi^{-1} - c_1  \id_{V_2},
	\end{align*}
or equivalently,
	\begin{equation} \label{eq:rho2=rho1+omega}
	\rho_2(s) = \phi \circ \rho_1(s) \circ \phi^{-1} + \omega \cdot \id_{V_2},
	\end{equation}
where $\omega = c_2 - c_1$. Squaring both sides of \eqref{eq:rho2=rho1+omega}, and using the fact that $s^2 = 1$, we get
	\[
	\id_{V_2} = {\id_{V_2}} + 2\omega \cdot \phi \circ \rho_1(s) \circ \phi^{-1} + \omega^2 \cdot \id_{V_2},
	\]
or equivalently, $\omega^2 \cdot \id_{V_2} = -2 \omega \cdot \phi \circ \rho_1(s) \circ \phi^{-1}$. The scalar $\omega$ does not depend on the choice of $s$, so if $\omega \neq 0$, we would deduce first for all $s$ in the set \eqref{eq:An-generators}, and then for all $s \in A_n$ by multiplicativity, that $\rho_1(s)$ is equal to a nonzero scalar multiple of $\id_{V_1}$. Since $\dim(V_1) \geq 2$ by assumption, this would contradict the irreducibility of $V_1$. Then $\omega = 0$, and \eqref{eq:rho2=rho1+omega} implies first for all $s$ in the set \eqref{eq:An-generators}, and then for all $s \in A_n$ by multiplicativity, that $\phi \circ \rho_1(s) = \rho_2(s) \circ \phi^{-1}$; that is, $\phi$ is an isomorphism of $\CA_n$-modules.
\end{proof}

\begin{corollary} \label{module-iso-equivalent}
Let $n \geq 5$, set $\g = \g_n$, and let $V_1$ and $V_2$ be two simple $\CS_n$-super\-modules. Then the following statements (in which `isomorphic' is taken to mean `isomorphic via a homogeneous isomorphism') are equivalent:
	\begin{enumerate}
	\item \label{item:Sn-iso-general} $V_1$ and $V_2$ are isomorphic as $\CS_n$-supermodules.
	\item \label{item:An-iso-general} $V_1$ and $V_2$ are isomorphic as $\CA_n$-supermodules.
	\item \label{item:gzero-iso-general} $V_1$ and $V_2$ are isomorphic as supermodules over the Lie subalgebra $\gzero \subseteq \CA_n$.
	\item \label{item:D(gzero)-iso-general} $V_1$ and $V_2$ are isomorphic as supermodules over the Lie subalgebra $\fD(\gzero) \subseteq \CA_n$.
	\end{enumerate}
\end{corollary}

\begin{proof}
The classification of the simple $\CS_n$-supermodules in Section \ref{subsec:group-algebra-Sn} shows that the simple $\CS_n$-supermodules are determined by their restrictions to $\CA_n$ and by the homogeneous degrees in which their simple $\CA_n$-factors are concentrated. Thus passing to the homogeneous subspaces of $V_1$ and $V_2$ (which are simple $\CA_n$-modules), the result follows by \cref{prop:irred-equivalent}.
\end{proof}

Let $V$ be a $\CA_n$-module with structure map $\rho: \CA_n \to \End(V)$. There are two evident ways to define an action of the Lie algebra $\fD(\gzero) \subseteq \CA_n$ on the dual space $V^*$. The first is the restriction from $\CA_n$ to $\fD(\gzero)$ of the group-theoretic dual module, $V^{*,\Grp}$, described before \cref{remark:self-dual}. The second module structure, which we denote $V^{*,\Lie}$, is via the contragredient action of a Lie algebra, defined for $x \in \fD(\gzero)$, $\phi \in V^*$, and $v \in V$ by $(x.\phi)(v) = -\phi(x.v)$. Fixing a basis for $V$ and the corresponding dual basis for $V^*$, the structure maps $\rho_{\Grp}$ and $\rho_{\Lie}$ for $V^{*,\Grp}$ and $V^{*,\Lie}$ are related to $\rho$ by $\rho_{\Grp}(x) = \rho(\imath(x))^T$ and $\rho_{\Lie}(x) = - \rho(x)^T$, where $\imath: \CA_n \to \CA_n$ is linear extension of the group inversion map $\sigma \mapsto \sigma^{-1}$, and $u^T$ is the transpose of $u$.

\begin{proposition} \label{prop:An-Lie-dual}
Let $n \geq 5$, set $\g = \g_n$, and let $V$ and $W$ be simple $\CA_n$-modules of dimension greater than $1$. Then as $\fD(\gzero)$-modules, $W \not\cong V^{*,\Lie}$.
\end{proposition}

\begin{proof}  We argue by induction on $n$.  First suppose $n = 5$. As discussed in Section \ref{S:Imageofgn}, we know for $n=5$ that \cref{theorem:finaltheorem} is true, and hence so is \cref{cor:D(g0)-Lie-subalgebra-CAn} (whose proof, for a given value of $n$, depends only on \eqref{eq:D(CSn)}, \cref{theorem:finaltheorem}, and \cref{cor:Wlambda(D)} for the same value of $n$). And by \cite[Theorem 2.5.15]{James:1981}, the trivial module is the unique one-dimensional $\CA_5$-module, and all other simple $\CA_5$-modules are of dimension $\geq 3$. There are now two cases: $V \cong W$ or $V \not\cong W$. In the first case we may assume that $V = W$. The structure maps $\rho: \fD(\gzero) \to \End(V)$ and $\rho_{\Lie}: \fD(\gzero) \to \End(V^{*,\Lie})$ have the same kernel, and hence both factor through the canonical projection $\fD(\gzero) \twoheadrightarrow \fsl(V)$ in \eqref{eq:D(g0)-Artin-Wedderburn}. Then $V \not\cong V^{*,\Lie}$ because the natural representation of $\fsl(V)$ is not self-dual if $\dim(V) \geq 3$. In the case $V \not\cong W$, we see from \cref{cor:D(g0)-Lie-subalgebra-CAn} that the image of the structure map $\fD(\gzero) \to \End(V) \oplus \End(W)$ has dimension $\dim(\fsl(V)) + \dim(\fsl(W))$. On the other hand, if $W \cong V^{*,\Lie}$, then up to a change of basis for $W$, the structure map would be of the form $x \mapsto (\rho(x),-\rho(x)^T)$, and hence its image would have dimension at most $\dim(\fsl(V))$. Then $W \not\cong V^{*,\Lie}$.

Now suppose $n \geq 6$ and that the claim is true for $n-1$. By \cref{remark:An-restriction}, $V$ and $W$ admit multiplicity-free restrictions to $A_{n-1}$, say, $V = \bigoplus_{i=1}^r V_i$ and $W = \bigoplus_{j=1}^s W_j$. By \cref{lemma:An-simple-implies-D(g0)-simple} and \cref{prop:irred-equivalent}, these are also decompositions of $V$ and $W$ into distinct simple modules for the Lie algebra $\fD_{n-1} := \fD((\g_{n-1})_{\zero})$. Then $V^{*,\Lie} = \bigoplus_{i=1}^r V_i^{*,\Lie}$ is a decomposition of $V^{*,\Lie}$ into distinct simple $\fD_{n-1}$-modules. Now a $\fD(\gzero)$-module isomorphism $W \cong V^{*,\Lie}$ will induce for each $j$ a $\fD_{n-1}$-module isomorphism $W_j \cong V_i^{*,\Lie}$ for some $i$. Since $n \geq 6$ and $\dim(W) > 1$, it follows from \cref{remark:An-restriction} and \cite[Theorem 2.5.15]{James:1981} that at least one of the $W_j$ is of dimension at least $3$. But then the isomorphism $W_j \cong V_i^{*,\Lie}$  contradicts the inductive assumption. Hence $W \not\cong V^{*,\Lie}$.
\end{proof}

\subsection{Structure of \texorpdfstring{$\g_n$}{gn}} \label{subsec:gn-structure}

In this section let $n \geq 2$, let $\g = \g_n$, and set $\fD = \fD(\gzero)$. For $\lambda \in \Pbar(n)$, let $\rho_\lambda: \fD \to \End(W^\lambda)_{\zero}$ be the $\fD$-module structure map. The image $\rho_\lambda(\fD) = W^\lambda(\fD)$ of $\rho_\lambda$ is given in \cref{cor:Wlambda(D)}. Set $\fD_\lambda = \ker(\rho_\lambda)$, and let $\fD^\lambda$ be the orthogonal complement of $\fD_{\lambda}$ with respect to the Killing form on $\fD$. Then $\fD^\lambda$ is an ideal in $\fD$, $\fD = \fD^\lambda \oplus \fD_\lambda$ as a Lie algebra, and $\rho_\lambda$ induces a Lie algebra isomorphism $\fD^\lambda \cong W^\lambda(\fD)$. Thus $\fD^\lambda$ is a simple ideal in $\fD$ (if $\lambda \in E_n$), or is uniquely expressible as a direct sum of two simple ideals in $\fD$ (if $\lambda \in F_n$).

\begin{proposition} \label{prop:Dlambda-Dmu-intersection}
Let $n \geq 5$, and let $\lambda,\mu \in \Pbar(n)$. If $\fD^\lambda \cap \fD^\mu \neq 0$, then $\lambda = \mu$.
\end{proposition}

\begin{proof}
Suppose $\fa = \fD^\lambda \cap \fD^\mu \neq 0$. There are several cases to consider:
	\begin{enumerate}
	\item \label{item:lambda-mu-En} $\lambda,\mu \in E_n$. Then $\fD^\lambda = \fa = \fD^\mu$, and the maps $\rho_\lambda$ and $\rho_\mu$ induce isomorphisms $\fa \cong W^\lambda(\fD)$ and $\fa \cong W^\mu(\fD)$.

	\item \label{item:lambda-En-mu-Fn} $\lambda \in E_n$ and $\mu \in F_n$. Then $\fD^\lambda = \fa$, and $\fa$ is one of the two simple ideals that comprise $\fD^\mu$. The map $\rho_\lambda$ induces an isomorphism $\fa \cong W^\lambda(\fD)$, while the map $\rho_\mu$ maps $\fa$ isomorphically into (precisely) one of the summands $\fsl(W^\mu_{\zero})$ or $\fsl(W^\mu_{\one})$ of $W^\mu(\fD)$. (There is also the symmetric case $\lambda \in F_n$ and $\mu \in E_n$, which we omit.)

	\item \label{item:lambda-mu-Fn-one-simple} $\lambda,\mu \in F_n$ and $\fa$ is a simple ideal. Then $\fa$ is one of the two simple ideals that comprise $\fD^\lambda$, and also for $\fD^\mu$. The map $\rho_\lambda$ sends $\fa$ isomorphically into (precisely) one of the summands $\fsl(W^\lambda_{\zero})$ or $\fsl(W^\lambda_{\one})$ of $W^\lambda(\fD)$, and similarly for $\rho_\mu$.

	\item \label{item:lambda-mu-Fn-two-simples} $\lambda,\mu \in F_n$ and $\fD^\lambda = \fa = \fD^\mu$. Then the maps $\rho_\lambda$ and $\rho_\mu$ induce isomorphisms $\fa \cong W^\lambda(\fD)$ and $\fa \cong W^\mu(\fD)$.
	\end{enumerate}
In each of the four cases, one must have $\dim(W^\lambda) = \dim(W^\mu)$. Then fixing homogeneous bases for $W^\lambda$ and $W^\mu$, we can make the identifications $W^\lambda = \C^{m|m} = W^\mu$ for some $m \in \N$, and we can interpret $\rho_\lambda$ and $\rho_\mu$ as Lie algebra homomorphisms
	\[
	\fD \to \End(\C^{m|m})_{\zero} = \End(\C^{m|0}) \oplus \End(\C^{0|m}) = \End(\C^m) \oplus \End(\C^m).
	\]

First consider case \eqref{item:lambda-mu-En}, in which $\lambda,\mu \in E_n$. Then the images of $\rho_\lambda$ and $\rho_\mu$ are each equal to the diagonal copy of $\fsl(\C^m)$ in $\End(\C^{m|m})_{\zero}$, and the composite induced map
	\[
	\fsl(\C^m) \xrightarrow{\rho_\lambda^{-1}} \fa \xrightarrow{\rho_\mu} \fsl(\C^m)
	\]
is a Lie algebra automorphism. Lie algebra automorphisms of $\fsl(\C^m)$ come in two forms:
	\begin{enumerate}
	\item $X \mapsto gXg^{-1}$ for some $g \in GL(\C^m)$, or
	\item $X \mapsto -(gX^t g^{-1})$ for some $g \in GL(\C^m)$, where $X^t$ denotes the transpose of $X$;
	\end{enumerate}
see \cite[IX.5]{Jacobson:1979}.  It should be noted that $X\mapsto -X^{t}$ is an automorphism of $\fsl(\C^m)$ that is not obtained by conjugation by $g \in GL(\C^m)$ when $n \geq 3$; again see \cite[IX.5]{Jacobson:1979}.  If $\rho_\mu \circ \rho_\lambda^{-1}$ is of the second form, then the $\fD$-module structure on $\C^{m|m}$ afforded by $\rho_\lambda$ is isomorphic to the dual of the $\fD$-module structure afforded by $\rho_\mu$, i.e., $W^\lambda \cong (W^\mu)^{*,\Lie}$ as $\fD$-modules. Passing to the homogeneous subspaces of $W^\lambda$ and $W^\mu$ (which are simple $\CA_n$-modules), this contradicts \cref{prop:An-Lie-dual}. Then $\rho_\mu \circ \rho_\lambda^{-1}$ must be of the first form, meaning the $\fD$-module structures on $\C^{m|m}$ afforded by $\rho_\lambda$ and $\rho_\mu$ are isomorphic, i.e., $W^\lambda \cong W^\mu$ as $\fD$-modules. Applying \cref{module-iso-equivalent} 
this implies that $W^\lambda \cong W^\mu$ as $\CS_n$-supermodules, and hence $\lambda = \mu$.

The reasoning for the other cases proceeds similarly. For example, in cases \eqref{item:lambda-En-mu-Fn} and \eqref{item:lambda-mu-Fn-one-simple}, one deduces that one of the $\fD$-module composition factors in $W^\mu$ is isomorphic to one (resp.\ both, if $\lambda \in E_n$) of the $\fD$-module composition factors 
in $W^\lambda$. In case \eqref{item:lambda-mu-Fn-two-simples}, one deduces that (both of) the $\fD$-module composition factors $W^\mu$ are isomorphic to the $\fD$-module composition factors in $W^\lambda$, perhaps up to 
parity change. In any case, 
\cref{prop:irred-equivalent} then implies that as $\CA_n$-modules, $W^\lambda$ and $W^\mu$ have simple constituents in common, which is only possible if $\lambda = \mu$. (In particular, cases \eqref{item:lambda-En-mu-Fn} and \eqref{item:lambda-mu-Fn-one-simple} are impossible.)
\end{proof}

\begin{corollary} \label{cor:D(gzero)-structure}
Let $n \geq 5$. Then $\fD = \bigoplus_{\lambda \in \Pbar(n)} \fD^\lambda$.
\end{corollary}

\begin{proof}
The sum $\sum_{\lambda \in \Pbar(n)} \fD^\lambda$ is a direct sum as a consequence of \cref{prop:Dlambda-Dmu-intersection}, and the sum is equal to all of $\fD$ as a consequence of the module structure map $\fD \to \bigoplus_{\lambda \in \Pbar(n)} \End(W^\lambda)$, $\sigma \mapsto \bigoplus_{\lambda \in \Pbar(n)} W^\lambda(\sigma)$, being faithful.
\end{proof}

\begin{theorem}\label{theorem:finaltheorem}
Let $n \geq 2$. Then $\g_n = \fD(\CS_n) + \C \cdot T_n$.
\end{theorem}

\begin{proof}
The theorem is true for $n \in \set{2,3,4,5}$ by \cref{lemma:gn-equality}, so we may assume that $n \geq 6$. We observed previously in \eqref{eq:gn-subset-derived-plus-Tn} that $\g_n \subseteq \fD(\CS_n) + \C \cdot T_n$ and that $T_n \in \g_n$, so we just need to show that $\fD(\CS_n) \subseteq \g_n$. Henceforward in this proof, we will let $\g = \g_n$, and we will identify $\CS_n$ with its image under the superalgebra isomorphism of \cref{cor:CSn-as-superalgebra}.

First we will show that $\fD(\CS_n)_{\one} \subseteq \g$. By \eqref{eq:D(CSn)}, one has
	\[
	\fD(\CS_n)_{\one} = \Big[ \bigoplus_{\lambda \in E_n} \sq(W^\lambda)_{\one} \Big] \oplus \Big[ \bigoplus_{\lambda \in F_n} \fsl(W^\lambda)_{\one} \Big],
	\]
and by \cref{cor:D(gzero)-structure}, one has, with notation as in \eqref{eq:Wlambda(D)},
	\begin{align*}
	\fD(\gzero) &= \Big[ \bigoplus_{\lambda \in E_n} \fsl(W_\lambda) \Big] \oplus \Big[ \bigoplus_{\lambda \in F_n} \fsl(W^\lambda_{\zero}) \oplus \fsl(W^\lambda_{\one}) \Big].
	\end{align*}
Then $\fD(\CS_n)_{\one}$ is a direct sum of pairwise non-isomorphic simple $\fD(\gzero)$-modules. Since $\gone$ is a $\fD(\gzero)$-submodule of $\fD(\CS_n)_{\one} + \C \cdot T_n$, it must contain some subset of the simple summands in $\fD(\CS_n)_{\one}$. Using \cref{theorem:Wlambda(gn)}, we see that each of these summands is contained in the image of the corresponding projection map $W^\lambda: \g \to \End(W^\lambda)$, and hence must have been contained in $\gone$. Thus $\fD(\CS_n)_{\one} \subseteq \g$.

Now applying \cref{lemma:lower-bound-dim} and \cref{lem:generated-by-odd}, we deduce that $\fsl(W^\lambda) \subseteq \g$ for each $\lambda \in F_n$, and we deduce that $\sq(W^\lambda) \subseteq \g$ for all $\lambda \in E_n$ with the exception of $\lambda = (n)$; recall \cref{conv:(n)-in-En}. For $\lambda = (n)$ the inclusion can be directly verified. For this partition, one has $\sq(W^{(n)}) = \C  \id_{W^{(n)}}$. If $\tau \in S_n$ is any transposition, then $1_{\CS_n} = \frac{1}{2}[\tau,\tau] \in \g$. But under the isomorphism of \cref{cor:CSn-as-superalgebra}, one has $1_{\CS_n} = \sum_{\lambda \in \Pbar(n)} \id_{W^\lambda}$, so
	\[
	\id_{W^{(n)}} = 1_{\CS_n} - \Big( \sum_{\substack{\lambda \in \Pbar(n) \\ \lambda \neq (n)}} \id_{W^\lambda} \Big) \in \g.
	\]
Thus $\sq(W^{(n)}) \subseteq \g$, and hence $\fD(\CS_n) \subseteq \g$.
\end{proof}

\begin{corollary} \label{cor:D(g0)-Lie-subalgebra-CAn}
Let $n \geq 2$, and let $\g = \g_n$. Then the Artin--Wedderburn Theorem isomorphism
	\begin{equation} \label{eq:CAn-Artin-Wedderburn}
	\CA_n \cong \Bigg[ \bigoplus_{\lambda \in E_n} \End(S^\lambda) \Bigg] \oplus \Bigg[ \bigoplus_{\lambda \in F_n} \End(S^{\lambda^+}) \oplus \End(S^{\lambda^-}) \Bigg]
	\end{equation}
restricts to a Lie algebra isomorphism
	\begin{equation} \label{eq:D(g0)-Artin-Wedderburn}
	\fD(\gzero) \cong \fD(\CA_n) = \Bigg[ \bigoplus_{\lambda \in E_n} \fsl(S^\lambda) \Bigg] \oplus \Bigg[ \bigoplus_{\lambda \in F_n} \fsl(S^{\lambda^+}) \oplus \fsl(S^{\lambda^-}) \Bigg].
	\end{equation}
\end{corollary}

\begin{proof}
Under the Artin--Wedderburn isomorphism, the Lie algebra $\fD(\gzero) \subseteq \gzero \subseteq \CA_n = (\CS_n)_{\zero}$ maps into the product of Lie algebras on the right-hand side of \eqref{eq:D(g0)-Artin-Wedderburn}, and one can then see that this map is a surjection by dimension comparison using \eqref{eq:D(CSn)}, \cref{theorem:finaltheorem}, and \cref{cor:Wlambda(D)}.
\end{proof}

\section{Declarations}\label{S:declarations}

\subsection{Ethical Approval}

Not Applicable.

\subsection{Funding}

CMD was supported in part by Simons Collaboration Grant for Mathematicians No.\ 426905.  JRK was supported in part by Simons Collaboration Grant for Mathematicians No.\ 525043.

\subsection{Availability of data and materials}

Not Applicable.

\makeatletter
\renewcommand*{\@biblabel}[1]{\hfill#1.}
\makeatother

\bibliographystyle{eprintamsplain}
\bibliography{lie-superalgebra-of-transpositions}

\end{document}